\documentclass[a4paper,11pt]{article}
\usepackage[top=3.0cm, bottom=3.0cm, inner=3.0cm, outer=3.0cm, includefoot]{geometry}

\usepackage[utf8]{inputenc}
\usepackage[T1]{fontenc}
\usepackage{authblk}

\usepackage{verbatim}
\usepackage{multirow}
\usepackage{multicol}
\usepackage{hyperref}

\usepackage{amssymb}
\usepackage{amsmath}
\usepackage{mathtools,bbm}
\usepackage{graphicx,tikz}
\usepackage{amsthm}
\usepackage{caption,subcaption}
\usepackage{float}
\usepackage{color}
\usepackage{enumerate}

\usepackage{cancel}

\newcommand{\eChar}{\begin{enumerate}[(i)]}
\newcommand{\eCharR}{\begin{enumerate}[(a)]}
\newcommand{\eBr}{\begin{enumerate}[(1)]}

\newcommand{\Abstract}

\theoremstyle{plain}
\newtheorem{lemma}{Lemma}[section]
\newtheorem{theorem}[lemma]{Theorem}
\newtheorem{proposition}[lemma]{Proposition}
\newtheorem{corollary}[lemma]{Corollary}

\theoremstyle{definition}

\newtheorem{conjecture}[lemma]{Conjecture}
\newtheorem{definition}[lemma]{Definition}
\newtheorem{remark}[lemma]{Remark}

\newtheorem{example}[lemma]{Example}

\numberwithin{equation}{section}

\title
{
Ricci curvature, diameter and eigenvalues of amply regular graphs
}

\author[1]{Kaizhe Chen\thanks{Email: ckz22000259@mail.ustc.edu.cn}}
\author[2]{Chunyang Hu\thanks{Email: chunyanghu@mail.ustc.edu.cn}}
\author[3]{Shiping Liu\thanks{Email: spliu@ustc.edu.cn}}
\author[4]{Heng Zhang\thanks{Email: hengz@mail.ustc.edu.cn}}
\affil[1]{School of the Gifted Young, University of Science and Technology of China, Hefei}
\affil[2,3,4]{School of Mathematical Sciences, University of Science and Technology of China, Hefei}

\date{ }

\begin{document}

\maketitle

\thispagestyle{plain}

\begin{abstract}
Amply regular graphs are graphs with local distance-regularity constraints. In this paper, we prove a weaker version of a conjecture proposed by Qiao, Park, and Koolen on diameter bounds of amply regular graphs and make new progress on Terwilliger's conjecture on finiteness of amply regular graphs. Terwilliger's conjecture can be considered as a natural extension of the Bannai-Ito conjecture about distance-regular graphs confirmed by Bang, Dubickas, Koolen, and Moulton. As a consequence, we show that there are only finitely many amply regular graphs with parameters $(n,d,\alpha,\beta)$ satisfying $\alpha\leq 6\beta-9$. We achieve these results by a significantly improved Lin--Lu--Yau curvature estimate and new Bakry--\'Emery curvature estimates. We further discuss applications of our curvature estimates to bounding eigenvalues, isoperimetric constants, and expansion properties. In addition, we obtain a volume estimate, which is sharp for hypercubes. 
\end{abstract}

\section{Introduction}

It was one of the main open problems about distance regular graphs to show that there are only finitely many of them with fixed vertex degree $d\geq 3$ \cite{Terwilliger82, Terwilliger83}. This was later known as the Bannai-Ito conjecture, referring to their 1984 book \cite[p.237]{BI} and has been confirmed in 2015 by Bang, Dubickas, Koolen and Moulton \cite{BDKM15}. The study of Bannai-Ito conjecture has particularly stimulated a very fruitful approach of bounding the diameter of a distance regular graph in terms of its vertex degree, see, e.g., \cite{BCN89,Hiraki07, Mulder79, NS22JCTB,NS22, Smith74,Terwilliger82, Terwilliger83}.
 
During the process of investigating the diameter bounds of distance regular graphs, it was realized that some diameter bounds hold even for graphs with only a fraction of the regularity properties distance-regular graphs possess. For example, Mulder \cite{Mulder79} obtained in 1979 a diameter bound for $(0,\lambda)$-graphs.
A $(0,\lambda)$-graph is a connected  graph such that any two distinct vertices have either $0$ or $\lambda$ common neighbors. 
Terwilliger \cite{Terwilliger83} derived in 1983 a diameter bound for $(s,c,a,k)$-graphs. In case of $s=2$, any $(2,c,a,k)$-graph is amply regular. An amply regular graph with parameters $(n, d,$ $\alpha, \beta)$ is a $d$-regular graph with $n$ vertices such that any two adjacent vertices have $\alpha$ common neighbors, and any two vertices at distance $2$ have $\beta$ common neighbors.

In this paper, we study diameter bounds for amply regular graphs with parameters $(n,d,\alpha,\beta)$ with $n$ being possibly infinity. Below we recall briefly several known diameter bounds and related open questions for an amply regular graph $G$.  
Inspired by the diameter bound due to Mulder for $(0,\lambda)$-graphs, Brouwer, Cohen and Neumaier \cite[Corollary 1.9.2]{BCN89} prove that 
\begin{equation}\label{eq:BCN}
    \mathrm{diam}(G)\leq \max\{3, d-2\beta+4\},\,\,\text{if}\,\,1\neq \beta\geq \alpha.
\end{equation}
Terwilliger's diameter bounds in \cite[Theorem 3.4]{Terwilliger83} imply particularly that 
\begin{align}
    &\mathrm{diam}(G)\leq d, \,\,&\text{if}\,\, 2=\beta\geq \alpha,\label{eq:Terwilliger1}\\
    &\mathrm{diam}(G)\leq \max\left\{5, \frac{2\beta(d-1)}{3\beta-2}-2\beta+7\right\},\,\,&\text{if}\,\, \beta\geq \max\{3,\alpha\}.\label{eq:Terwilliger2}
\end{align}
Notice that the diameter bound \eqref{eq:Terwilliger2} is stronger than \eqref{eq:BCN} for large $d$. Particularly, we point out that the coefficient of $d$ in \eqref{eq:Terwilliger2} satisfies $1>\frac{2\beta}{3\beta-2}>\frac{2}{3}$.
Terwilliger \cite[p.160]{Terwilliger83} further proposed the following conjecture. 
\begin{conjecture}[Terwilliger]\label{conj:Terwilliger}
    Any amply regular graph with parameters $(n,d,\alpha,\beta)$ with $\beta\geq 2$ is always finite. 
\end{conjecture}
We mention that both the diameter bounds \eqref{eq:Terwilliger1}-\eqref{eq:Terwilliger2} and the conjecture of Terwilliger are made for $(s,c,a,k)$-graphs. We reduce them to amply regular graphs here. Terwilliger's diameter bounds actually confirm Conjecture \ref{conj:Terwilliger} for amply regular graphs with $1\neq \beta\geq \alpha$.  

 Qiao, Park, and Koolen \cite{QPK19} observed from Terwilliger's diameter bounds that any connected amply regular graph with parameters $(n,d,\alpha,\beta)$ such that $\beta>\frac{d}{3}\geq 1$ and $\beta\geq \alpha$ has diameter at most $7$. Based on this phenomenon, they \cite[Conjecture 1.2]{QPK19} proposed the following conjecture.

 \begin{conjecture}[Qiao, Park and Koolen] \label{conj:QPK} Let $G$ be a connected amply regular graph with parameters $(n,d,\alpha,\beta)$ and $0<\varepsilon<1$ be a real number. Then there exists a constant $K(\varepsilon)$ such that, if $\beta>\varepsilon d\geq 1$ and $\beta\geq \alpha$ all hold, then the diameter of $G$ is at most $K(\varepsilon)$.
 \end{conjecture}

Recently, Huang, the third named author and Xia \cite{HLX24} show for an amply regular graph $G$ with parameters $(n,d,\alpha,\beta)$ that 
\begin{equation}\label{eq:HLX24}
    \mathrm{diam}(G)\leq \left\lfloor\frac{2d}{3}\right\rfloor, \,\,\text{if}\,\,\beta>\alpha\geq 1.
\end{equation}

In this paper, we establish the following diameter bound for amply regular graphs. 
\begin{theorem}\label{thm:main1}
    Let $G$ be an amply regular graph with parameters $(n,d,\alpha,\beta)$. Then the diameter of $G$ satisfies
    \begin{equation}\label{eq:main1}
 \mathrm{diam}(G)\leq \left\lfloor\frac{2d}{2+\max\left\{\frac{\alpha}{2},\left\lceil\frac{\beta-\alpha}{\beta-1}\alpha\right\rceil\right\}}\right\rfloor,\,\,\text{if}\,\,\beta\geq \max\{3,\alpha\}. 
    \end{equation}
\end{theorem}
We remark that \eqref{eq:main1} also holds for the case $2=\beta>\alpha$. That is, \eqref{eq:main1} actually holds for the case $1\neq \beta\geq \alpha$ except for $\alpha=\beta=2$. 
We conjecture that the bound \eqref{eq:main1} still holds when $\alpha=\beta=2$.

Theorem \ref{thm:main1} significantly improves previous diameter bounds \eqref{eq:BCN},\eqref{eq:Terwilliger2} and \eqref{eq:HLX24} (for large $d$). Indeed, the coefficient of $d$ in \eqref{eq:main1} can be much smaller than $\frac{2}{3}$. 

Moreover, Theorem \ref{thm:main1} confirms a weaker version of Conjecture \ref{conj:QPK}.
\begin{corollary}\label{cor:QPK_weak}
    Let $G$ be a connected amply regular graph with parameters $(n,d,\alpha,\beta)$ and $0<\varepsilon<1$ be a real number. If $\beta\geq \alpha >\varepsilon d\geq 1$, then the diameter of $G$ is at most $4/\varepsilon$.
\end{corollary}

We prove Theorem \ref{thm:main1} by discrete Ricci curvature estimates of amply regular graphs. Recall that the diameter bound \eqref{eq:HLX24} is obtained via estimating the Lin--Lu--Yau curvature \cite[Theorem 1.3]{HLX24}. The proof of Theorem \ref{thm:main1} is built upon two main ingredients: a significant improvement of the Lin--Lu--Yau curvature estimate and a formula to calculate the Bakry--\'Emery curvature of amply regular graphs; The former is related to local matching conditions and is achieved by using K\"onig theorem for regular bipartite graphs with possibly multiple edges, while the latter is related to the adjacency eigenvalue of the local graph induced by neighbors of a vertex and is a generalization of previous results for strongly regular graphs in the arXiv version of \cite{CLP}.

Lin--Lu--Yau curvature \cite{LLY11,O09} and Bakry--\'Emery curvature \cite{BE,LY10} of graphs are discrete analogues of Ricci curvature of Riemannian manifolds, based on optimal transportation and Bochner identity, respectively. The study of discrete Ricci curvatue has attracted lots of attention in recent years, see, e.g., the book \cite{NR17} and references therein. Various combinatorial, geometric and analytic consequences of Lin--Lu--Yau and Bakry--\'Emery curvature lower bounds have been established, see, e.g., \cite{BRT21,CKKLMP20,HPS24,HL17,KKRT,LMP24,LW20,MWAdvMath,SalezGAFA,SalezJEMS,S20,Smith14}. Especially, positive Lin--Lu--Yau or Bakry--\'Emery curvature lower bound implies the finiteness of a locally finite graph and a sharp diameter bound \cite{LLY11,LMP18}. Such results are considered as discrete analogues of Bonnet-Myers theorem from Riemannian geometry. 

Another consequence of our discrete Ricci curvature approach is an improvement of Terwilliger's finiteness result for amply regular graphs, concerning his Conjecture \ref{conj:Terwilliger}.
\begin{theorem}\label{thm:finite_ARG}
Any amply regular graph with parameters $(n,d,\alpha,\beta)$ with $\alpha\leq6\beta-9$ is finite. Moreover, 
    there are only finitely many amply regular graphs with $\alpha\leq6\beta-9$ and fixed degree.
\end{theorem}
Recall that Terwilliger \cite{Terwilliger83} has shown the above result for the case $1\neq\beta\geq \alpha$. Theorem \ref{thm:finite_ARG} is proved via Bakry-\'Emery curvature lower bound estimates not only for the case $\beta\geq \alpha$ but also for certain case with $\alpha>\beta$.

Our discrete Ricci curvature approach further leads to spectral gap estimates for amply regular graphs. Let $\theta_1\leq\cdots\leq \theta_{n-1}\leq \theta_n$ be the adjacency eigenvalues of a $d$-regular graph $G$ with $n$ vertices. Then we have 
\[-d\leq \theta_1\leq\cdots\leq \theta_{n-1}\leq \theta_n=d,\]
where $\theta_{n-1}=d$ if and only if $G$ is disconnected, and $\theta_1=-d$ if and only if $G$ is bipartite.

\begin{theorem}\label{thm:upper_bound_for_adjacency_eigencalue}
    Let $G$ be an amply regular graph with parameters $(n,d,\alpha,\beta)$. 
    \begin{itemize}
        \item [(i)] If $1\neq \beta\geq \alpha$ and $(\alpha,\beta)\neq (2,2)$, then 
        \[\theta_{n-1}\leq d-2-\max\left\{\frac{\alpha}{2},\left\lceil\frac{(\beta-\alpha)\alpha}{\beta-1}\right\rceil\right\}.\]
        \item [(ii)] If $\alpha=\beta=2$, then 
        \[\theta_{n-1}\leq d-2.\]
    \end{itemize}
\end{theorem}

This improves the previous sharp eigenvalue estimates in \cite[Corollary 1.6]{HLX24} significantly. Notice that $n<\infty$ in Theorem \ref{thm:upper_bound_for_adjacency_eigencalue} is guaranteed by the parameter conditions due to Theorem \ref{thm:finite_ARG}.

Moreover, our approach also produces an estimate for the gap between $\theta_1$ and $-d$ in the other end. 

\begin{theorem}\label{thm:dual}
 Let $G$ be an amply regular graph with parameters $(n,d,\alpha,\beta)$ of diameter at least $4$. If $2\le\alpha\le10\beta-12$, then 
 \[\theta_1\geq -d+2.\]
\end{theorem}
        
        

We achieve this via establishing a formula to calculate the Bakry--\'Emery curvature of a signed graph $(G,\sigma)$ with $\sigma$ being the all-$-1$ signature. A signed graph is a graph $G=(V,E)$ associated with a signature $\sigma: E\to \{+1,-1\}$ \cite{Harary,Zaslavsky}. Bakry--\'Emery curvature for general connection graphs, including signed graphs as a special case, is introduced by the third named author, M\"unch and Peyerimhoff \cite{LMP}, and further studied in \cite{CL24,HL22}.

Our eigenvalue estimates in Theorem \ref{thm:upper_bound_for_adjacency_eigencalue} and Theorem \ref{thm:dual} naturally yields estimates on Cheeger constant, bipartiteness constant \cite{BJ13,Trevisan}, and expansion constant \cite[p.151]{AS16} of an amply regular graph, see Corollary \ref{cor:CheegerDual} and Theorem \ref{thm:expander}.

We further derive the following sharp volume growth estimate by combining our discrete curvature estimates with two key lemmas due to Lin-Lu-Yau \cite[Lemma 4.4]{LLY11} and Terwilliger \cite[Lemma 3.3]{Terwilliger83}, respectively.
\begin{theorem}\label{thm:volume}
    Let $G=(V,E)$ be an amply regular graph with parameters $(n,d,\alpha,\beta)$ such that $1 \neq \beta \geq \alpha$. Then we have for any $x \in V$ and $i\geq1$ that   $$|S_{i+1}(x)|\leq \frac{d-\max\left\{\alpha+1, i\left(1+\frac{1}{2}\left\lceil\frac{\alpha(\beta-\alpha)}{\beta-1}\right\rceil\right)\right\}}{\beta+i-1}|S_{i}(x)|,$$
    where $S_i(x)$ is the set of vertices of distance $i$ to $x$.
\end{theorem}
The volume growth estimates in Theorem \ref{thm:volume} and the corresponding volume estimate are sharp for hypercube graphs (see Example \ref{ex:hypercube}).

Finally, we comment that our diameter and eigenvalue estimates in Theorems \ref{thm:main1}, \ref{thm:upper_bound_for_adjacency_eigencalue} and \ref{thm:dual} still hold under more general conditions, see Corollary \ref{cor:BE_d} and Theorem \ref{thm:eigenvalueestimate}.

The rest of our paper is organized as follows. In Section \ref{preliminary}, we provide basics about the two discrete Ricci curvature notions and about signed and amply regular graphs. We establish the Lin--Lu--Yau curvature and Bakry--\'Emery curvature estimates of amply regular graphs with applications to diameter bounds in Sections \ref{section:LLY} and \ref{section:BE}, respectively. We discuss further applications of our curvature estimates in bounding eigenvalues, isoperimetric constants and volume growth of amply regular graphs in Section \ref{section:furtherApp}.

\section{Preliminaries}\label{preliminary}
In this section, we collect preliminaries about Lin--Lu--Yau curvature, Bakry--\'Emery curvature, signed graphs and amply regular graphs. Let $G=(V,E)$ be a graph with vertex set $V$ and edge set $E$. For any two vertices $x,y\in V$, we write $xy\in E$ if there exists an edge $\{x,y\}\in E$ between them. The vertex degree $d_x$ of a vertex $x\in V$ is defined by $d_x=\sum_{y\in V:\, xy\in E} 1$. A graph $G$ is called locally finite if $d_x<\infty$ for any $x\in V$.

\subsection{Lin--Lu--Yau Curvature}

 \begin{definition}[Wasserstein distance]
   
     Let $G=(V,E)$ be a locally finite graph, $\mu_1$ and $\mu_2$ be two probability measures on $V$. The Wasserstein distance $W_1(\mu_1, \mu_2)$ between $\mu_1$ and $\mu_2$ is defined as
     \[W_1(\mu_1,\mu_2)=\inf_{\pi}\sum_{y\in V}\sum_{x\in V}d(x,y)\pi(x,y),\]
     where $d(x,y)$ is the combinatorial distance between $x$ and $y$ in $G$, and the infimum is taken over all maps $\pi: V\times V\to [0,1]$ satisfying
     \[\mu_1(x)=\sum_{y\in V}\pi(x,y),\,\,\mu_2(y)=\sum_{x\in V}\pi(x,y).\] Such a map is called a transport plan.
     \end{definition}
     For any $p\in [0,1]$, consider the particular probability measure $\mu_x^p$ around a vertex $x\in V$ defined as follows:
     \[\mu_x^p(y)=\left\{
                    \begin{array}{ll}
                      p, & \hbox{if $y=x$;} \\
                      \frac{1-p}{d_x}, & \hbox{if $xy\in E$;} \\
                      0, & \hbox{otherwise.}
                    \end{array}
                  \right.
     \]

   The following definition of discrete curvature is introduced by Ollivier \cite{O09} and Lin, Lu and Yau \cite{LLY11}, as a discrete analgoue of the Ricci curvature of a Riemannian manifold.
   \begin{definition}[$p$-Ollivier curvature and Lin--Lu--Yau curvature \cite{LLY11,O09}] Let $G=(V,E)$ be a locally finite graph. For any vertices $x,y\in V$ and $p\in [0,1]$, the $p$-Ollivier curvature $\kappa_p(x,y)$ is defined as
       \[\kappa_p(x,y):=1-\frac{W_1(\mu_x^p,\mu_y^p)}{d(x,y)}.\]
     We then define the Lin--Lu--Yau curvature $\kappa_{LLY}(x,y)$ as
       \[\kappa_{LLY}(x,y):=\lim_{p\to 1}\frac{\kappa_p(x,y)}{1-p}.\]
       \end{definition}
       Notice that $\kappa_1(x,y)$ is always $0$. Hence, the Lin--Lu--Yau curvature $\kappa_{LLY}(x,y)$ equals to the negative of left derivative of the function $p\mapsto \kappa_p(x,y)$ at $p=1$.
       
For regular graphs, Bourne et. al. \cite{BCMLP} found the following limit-free reformulation of the Lin--Lu--Yau curvature, by showing that the curvature function $p\mapsto \kappa_p(x,y)$ at any edge $xy\in E$ is piecewise linear with at most $2$ linear parts.
       \begin{theorem}[{\cite{BCMLP}}]\label{thm:BCMLP}
           Let $G=(V,E)$ be a $d$-regular locally finite graph. For any edge $xy\in E$, we have
\begin{equation*}
\kappa_{LLY}(x,y)=2\kappa_{\frac{1}{2}}(x,y)=\frac{d+1}{d}\kappa_\frac{1}{d+1}(x,y).
\end{equation*}
       \end{theorem}
We have the following Lin--Lu--Yau curvature upper bound and an interesting connection to the existence of a local perfect matching, see, e.g., \cite[Propsition 2.7]{CKKLMP20}. The upper bound estimate has been established in \cite[Lemma 4.4]{LLY11}, and its relations with local matchings have been discussed in \cite{Bonini,MWAdvMath,Smith14}.
\begin{theorem}[{\cite[Propsition 2.7]{CKKLMP20}}]\label{thm:upperk}
   Let $G=(V,E)$ be a $d$-regular graph. For any edge $xy\in E$, we have
   \[\kappa_{LLY}(x,y)\leq \frac{2+|\Delta_{xy}|}{d},\]
   where $\Delta_{xy}:=S_1(x)\cap S_1(y)$ and $S_1(v):=\{w\in V|\, vw\in E\}$.  Moreover, the following two propositions are equivalent:
   \begin{itemize}
   \item[(a)] $\kappa_{LLY}(x,y) = \frac{2+|\Delta_{xy}|}{d}$;
   \item[(b)] there is a perfect matching between  $N_x$ and $N_y$, where
     $N_x:= S_1(x) \backslash (\Delta_{xy} \cup \{y\})$ and $N_y:= S_1(y) \backslash (\Delta_{xy} \cup \{x\})$.
   \end{itemize}
   \end{theorem}

Next, we recall K\"onig's theorem which would be used in our curvature estimate.

\begin{theorem}[König's theorem]\label{konig}
A bipartite graph $G$ can be decomposed into $d$ edge-disjoint perfect matchings if and only if $G$ is $d$-regular.
\end{theorem}

Notice that in Theorem \ref{konig}, the bipartite graph is allowed to have multiple edges. 

Let us conclude this subsection by recalling the following diameter estimate, which is an analogue of the Bonnet--Myers theorem in Riemannian geometry.
\begin{theorem}[Discrete Bonnet--Myers theorem via Lin--Lu--Yau curvature \cite{LLY11,O09}]\label{LLYBM} Let $G=(V,E)$ be a locally finite connected graph. If $\kappa_{LLY}(x,y)\geq k>0$ holds true for any edge $xy\in E$, then the graph is finite and the diameter satisfies
\[\mathrm{diam}(G)\leq \frac{2}{k}.\]
\end{theorem}

\subsection{Bakry--\'Emery curvature of signed graphs}
 A signed graph $(G,\sigma)$ is a graph $G=(V,E)$ associated with a signature $\sigma:E\rightarrow \{+1,-1\}$. The signature of a cycle in $G$ is defined to be the product of all signatures on edges of this cycle. 
A signed graph $(G,\sigma)$ is called \emph{balanced} if the signature of each cycle of $G$ is $+1$, and \emph{anti-balanced} if the signature of each even cycle is $+1$ and of each odd cycle is $-1$ \cite{Harary}. 

An important characterization of balanced and anti-balanced signatures is given via the following concept of switching equivalence.
\begin{definition}[Switching equivalence]
    Let $G=(V,E)$ be a graph. Two signatures $\sigma,\sigma'$ of $G$ are said to be \emph{switching equivalent} if there exists a function $\tau:V\rightarrow \{+1,-1\}$, called a \emph{switching function}, such that for each edge $xy\in E$,
    $$\sigma'(xy)=\tau(x)\sigma(xy)\tau(y).$$
\end{definition}

Indeed, a signed graph $(G,\sigma)$ is balanced if and only if the signature $\sigma$ is switching equivalent to the all-$+1$ signature and is anti-balanced if and only if the signature is switching equivalent to the all-$-1$ signature \cite{Zaslavsky}. 

For an edge $xy\in E$, We will write $\sigma_{xy}=\sigma(xy)$ for short.

Next, we recall the connection Laplacian and the corresponding Bakry--\'Emery curvature on a signed graph \cite{BE,LMP}. The connection Laplacian $\Delta^\sigma$ is defined as below. For any function $f:V\rightarrow \mathbb{R}$ and any vertex $x\in V$, 
$$\Delta^{\sigma}f(x):=\sum_{y:\,xy\in E}
(\sigma_{xy}f(y)-f(x)).$$
In matrix form, we have $\Delta^\sigma=A^\sigma-D$, where $A^\sigma$ is the signed adjacency matrix and $D$ is the diagonal degree matrix. In case that $\sigma\equiv +1$ is the all-$+1$ signature, we write $\Delta$ instead of $\Delta^\sigma$ for short.

Following \cite{LMP}, we define the $\Gamma^\sigma$-operator and $\Gamma^\sigma_{2}$-operator as follows. For any functions $f,g:V\rightarrow \mathbb{R}$ and any vertex $x\in V$, we have
\begin{align*}
2\Gamma^\sigma(f,g)(x):=&\Delta(fg)(x)-(\Delta^\sigma f)g(x)-f(\Delta^\sigma g)(x).\\
2\Gamma_{2}^{\sigma}(f,g)(x):=&\Delta\Gamma^\sigma(f,g)(x)-\Gamma^\sigma(\Delta^\sigma f,g)(x)-\Gamma^\sigma(f,\Delta^\sigma g)(x).
\end{align*}
Notice in the above definitions that both the connection Laplacians $\Delta^\sigma$ for the signature $\sigma$ and $\Delta$ for the all-$+1$ signature are involved.
For simplicity, we denote
    \[\Gamma^{\sigma}(f)(x):=\Gamma^{\sigma}(f,f)(x),\,\,
    \Gamma^{\sigma}_{2}(f)(x):=\Gamma^{\sigma}_{2}(f,f)(x).\]

\begin{definition}[Bakry--\'Emery curvature \cite{BE,LMP}] Let $(G,\sigma)$ be a signed graph. We say a vertex $x$ satisfies the Bakry--\'Emery curvature dimension inequality $CD^{\sigma}(K,N)$ for $K\in \mathbb{R}$ and $N\in (0,\infty]$, if 
 it holds for any function $f:V\rightarrow{\mathbb{R}}$ that
    \begin{equation}\label{curvatureineq}
    \Gamma_{2}^\sigma(f)(x)\geq\frac{1}{N}(\Delta^\sigma f)^{2}(x)+K\Gamma^\sigma(f)(x).
    \end{equation}
For $N\in (0,\infty]$, we define the $N$-Bakry--\'Emery curvature at a vertex $x$ of $(G,\sigma)$ as
    \[K_{BE}(G,\sigma, x,N):=\sup\{K: CD^\sigma(K,N) \text{ holds at } x\}.\]
\end{definition}
We are particularly interested in the case $N=\infty$, for which we use the convention that $\frac{1}{N}=0$. 

\begin{proposition}[Switching invariance \cite{LMP}]\label{prop:switching}
Let $(G,\sigma)$ be a signed graph with $G=(V,E)$. Let $\tau:V\to \{+1,-1\}$ be a switching function and $\sigma^\tau$ be the signature such that
$\sigma^{\tau}_{xy}:=\tau(x)\sigma_{xy}\tau(y)$ for any $xy\in E$.
Then $(G,\sigma)$ satisfy $CD^{\sigma}(K,N)$ at $x$ if and only if $(G,\sigma^\tau)$ satisfy $CD^{\sigma^\tau}(K,N)$ at $x$. That is, we have 
$$K_{BE}(G,\sigma,x,N)=K_{BE}(G,\sigma^\tau,x,N).$$
\end{proposition}

In this paper, we are concerned with the $\infty$-Bakry--\'Emery curvature of a graph associated with either the all-$+1$ or the all-$-1$ signature. We will write the corresponding Bakry--\'Emery curvature as
\[K_{BE}(+,x):=K(G,+,x,\infty),\,\,\text{and}\,\,K_{BE}(-,x):=K_{BE}(G,-,x,\infty),\]
for short.  

The following diameter estimate via Bakry--\'Emery curvature is another analogue of the Bonnet-Myers theorem in Riemannian geometry.
\begin{theorem}[Discrete Bonnet--Myers theorem via Bakry--\'Emery curvature \cite{LMP18}]\label{BEBM} Let $G=(V,E)$ be a locally finite connected graph. If $K_{BE}(+,x)\geq K>0$ holds true for any vertex $x\in V$, then the graph is finite and the diameter satisfies
\[\mathrm{diam}(G)\leq \frac{2d_{\max}}{K},\]
where $d_{\max}$ stands for the maximal vertex degree.
\end{theorem}
\begin{remark}
    We point out that similar diameter estimate as above using  Bakry--\'Emery curvature of a general signature $\sigma$ dose NOT hold any more \cite[Example 8.1]{HL22}. 
    On the other hand, Bakry--\'Emery curvature on general signed graphs is quite useful for studying eigenvalue estimates and heat semigroup properties \cite{LMP}. 
\end{remark}

The two discrete Ricci curvature notions, Bakry--\'Emery curvature and Lin--Lu--Yau curvature, are naturally related. However, they can also be very different. There are special cases in which the two curvature values have opposite signs. \cite{CKKLP} shows more comparison results between these two curvature notions.

\subsection{Amply regular graphs}
We recall the definition of an amply regular graph from \cite[Section 1.1]{BCN89}.
       \begin{definition}[Amply regular graph \cite{BCN89}]\label{ARG} Let $G$ be a $d$-regular graph with $n$ vertices. Then $G$ is called an amply regular graph with parameters $(n, d,$ $\alpha, \beta)$ if any two adjacent vertices have $\alpha$ common neighbors, and any two vertices with distance $2$ have $\beta$ common neighbors.
         \end{definition}
In this paper, we allow $n=\infty$ and restrict ourselves to connected amply regular graph, which is neither complete nor empty. Then, we naturally have $d\geq \alpha+2$. The following degree estimates will also be useful for our later purpose.
\begin{theorem}[{\cite[Theorem 1.2.3 and Corollary 1.2.4]{BCN89}}]\label{thm:1.2.3}
    Let $G$ be a connected amply regular graph with parameters $(n,d,\alpha,\beta)$. If $G$ has diameter at least three and $d<3(\alpha-\beta+2)$, then $G$ is the icosahedron or the line graph of a regular graph of girth at least $5$.
\end{theorem}
\begin{theorem}[{\cite[Theorem 1.9.3]{BCN89}}]\label{thm:1.9.3}
    Let $G$ be a connected amply regular graph of diameter at least $4$ with parameters $(n,d,\alpha,\beta)$, then $d\ge 2\beta$.
\end{theorem}
\section{Lin--Lu--Yau curvature of amply regular graphs}\label{section:LLY}
We first prove the following result.
\begin{theorem}\label{thm:beta_alpha_1}
Let $G=(V,E)$ be an amply regular graph with parameters $(n,d,\alpha,\beta)$ such that $1\neq\beta\geq \alpha$. Then the Lin--Lu--Yau curvature $\kappa_{LLY}(x,y)$ of any edge $xy\in E$ satisfies
\begin{equation}\label{eq:beta_alpha_1}
    \kappa_{LLY}(x,y)\geq\frac{2+\left\lceil\frac{\alpha(\beta-\alpha)}{\beta-1}\right\rceil}{d}.
\end{equation}
\end{theorem}
We first explain the notations. We set $S_1(v):=\{w\in V|\, vw\in E\}$ and use the notations:$$\Delta_{xy}:=S_1(x)\cap S_1(y),\ 
     N_x:= S_1(x) \backslash (\Delta_{xy} \cup \{y\})\ {\rm and}\ N_y:= S_1(y) \backslash (\Delta_{xy} \cup \{x\}).$$
It follows by definition that $|\Delta_{xy}|=\alpha\ {\rm and}\ |N_x|=|N_y|=d-\alpha-1$.

As a preparation for the proof of Theorem \ref{thm:beta_alpha_1}, we construct from the local structure of the edge $xy\in E$ a bipartite multigraph $H_G = (V_H, E_H)$ as follows. Let us denote the $\alpha$ vertices in $\Delta_{xy}$ by $z_1, \cdots, z_\alpha$. The vertex set of $H_G$ is given by $$V_H=N_x\cup N_y\cup \Delta_{xy}\cup \Delta'_{xy}.$$Here $\Delta'_{xy}:=\{z'_1, \cdots, z'_\alpha\}$ is a new added set with $\alpha$ vertices, which is considered as a copy of $\Delta_{xy}$. The edge set $E_H:=\cup_{i=1}^5E_i$ is given by
\begin{align}\notag
&E_1=\{vw|v\in N_x, w\in N_y, vw\in E\},\\ \notag
&E_2=\{vz'_i|v\in N_x, z'_i\in \Delta'_{xy}, vz_i\in E\},\\ \notag
&E_3=\{z_iw| z_i\in \Delta_{xy},w\in N_y, z_iw\in E\},\\ \notag
&E_4=\{z_iz'_j| z_iz_j\in E,1\le i\le \alpha,1\le j\le \alpha \},\\ \notag
&E_5=\{e_i^j|e_i^j=z_iz'_i, 1\le i\le\alpha, 1\le j\le \beta-\alpha\},
\end{align}
where $E_5$ contains $\beta-\alpha$ number of parallel edges between $z_i$ and $z'_i$ for each $1\le i\le\alpha$. Notice that $\Delta_{xy}=\emptyset$ if $\alpha=0$ and $E_5=\emptyset$ if $\beta=\alpha$.
\begin{lemma}
$H_G$ is $(\beta-1)$-regular.
\end{lemma}
\begin{proof}
For any $v\in N_x$, we have $d(v, y) = 2$ in $G$.  Therefore, there are $\beta$ common neighbors of $v$ and $y$ in $V$, $\beta-1$ vertices of which lie in $N_y\cup \Delta_{xy}$. Thus, in the graph $H_G$, there are $\beta-1$ neighbors of $v$ in $N_y\cup \Delta'_{xy}$ by definition. Similarly, for any $w\in N_y$, there are $\beta-1$ neighbors of $w$ in $N_x\cup \Delta_{xy}$.

For any $z_i\in \Delta_{xy}$, there are $\alpha$ common neighbors of $z_i$ and $y$ in $V$, $\alpha-1$ vertices of which lie in $N_y\cup \Delta_{xy}$. Together with $\beta-\alpha$ parallel edges between $z_i$ and $z'_i$, there are $\beta-1$ neighbors of $z_i$ in $N_y\cup \Delta'_{xy}$ in the graph $H_G$. Similarly, for any $z'_i\in\Delta'_{xy}$, there are $\beta-1$ neighbors of $z'_i$ in $N_x\cup \Delta_{xy}$.
\end{proof}
\begin{proof}[Proof of Theorem \ref{thm:beta_alpha_1}]
By Theorem \ref{konig}, $E_H$ can be decomposed into $\beta-1$ edge-disjoint perfect matchings. Since $|E_5|=\alpha(\beta-\alpha)$, there exists a perfect matching $\mathcal M$ such that 
\begin{equation*}\label{ME}
    |\mathcal M\cap E_5|\ge \left\lceil\frac{\alpha(\beta-\alpha)}{\beta-1}\right\rceil.
\end{equation*}
We consider the following particular transport plan $\pi_0:V\times V\to [0,1]$ from $\mu_x^{\frac{1}{d+1}}$ to $\mu_y^{\frac{1}{d+1}}$:

\begin{center}
$\pi_0(v,w)=\begin{cases}
\frac{1}{d+1}, &{\rm if}\ v\in N_x\cup \Delta_{xy}, w\in N_y\ {\rm and}\ vw\in {\mathcal M};\\
\frac{1}{d+1}, &{\rm if}\ v\in N_x\cup \Delta_{xy},w\in \Delta_{xy}\ {\rm and}\ vw'\in {\mathcal M};\\
0, &{\rm otherwise}.
\end{cases}$
\end{center}

It is then direct to check that $\pi_0$ is indeed a transport plan. There are $|{\mathcal M}|$ pairs of $(v,w)$ such that $\pi_0(v,w)\ne 0$. Among them, there are $|{\mathcal M}\cap E_5|$ pairs with $d(v,w)=0$ and $|{\mathcal M}|-|{\mathcal M}\cap E_5|$ pairs with $d(v,w)=1$. Therefore, we have
\begin{align}\notag
W\left(\mu_x^{\frac{1}{d+1}},\mu_y^{\frac{1}{d+1}}\right)&\le \sum_{v\in V}\sum_{w\in V}d(v,w)\pi_0(v,w)\\ \notag
&= \frac{1}{d+1}(|{\mathcal M}|-|{\mathcal M}\cap E_5|)\\ \notag
&\le \frac{1}{d+1}\left(d-1-\left\lceil\frac{\alpha(\beta-\alpha)}{\beta-1}\right\rceil\right).
\end{align}
It follows by Theorem \ref{thm:BCMLP} that \begin{equation*}\label{3.2}
    \kappa_{LLY}(x,y)=\frac{d+1}{d}\left(1-W\left(\mu_x^{\frac{1}{d+1}},\mu_y^{\frac{1}{d+1}}\right)\right) \ge \frac{2+\left\lceil\frac{\alpha(\beta-\alpha)}{\beta-1}\right\rceil}{d},
\end{equation*}
completing the proof.
\end{proof}

\begin{corollary}\label{LLYd}
Let $G=(V,E)$ be an amply regular graph with parameters $(n,d,\alpha,\beta)$ such that $1\neq\beta\geq \alpha$. Then, the diameter of $G$ satisfies\[\mathrm{diam}(G)\leq \left\lfloor\frac{2d}{2+\left\lceil\frac{\beta-\alpha}{\beta-1}\alpha\right\rceil}\right\rfloor.\]
\end{corollary}
\begin{proof}
It is a direct consequence of Theorem \ref{LLYBM} and Theorem \ref{thm:beta_alpha_1}.
\end{proof}

\begin{remark}
It has been proved that $\kappa_{LLY}(x,y)\geq \frac{2}{d}$ if $\beta=\alpha>1$ \cite[Theorem 1.3]{LL21}, and $\kappa_{LLY}(x,y)\geq \frac{3}{d}$ if $\beta>\alpha\geq 1$ \cite[Theorem 1.3]{HLX24}.
    Theorem \ref{thm:beta_alpha_1} recovers and improves those previous results. In particular, for given $\alpha$ and $d$, our lower bound tends to the maximal possible curvature value $\frac{2+\alpha}{d}$ when $\beta$ is large enough.
\end{remark}

\begin{corollary}\label{cor:perfect_matching}
    Let $G=(V,E)$ be an amply regular graph with parameters $(n,d,\alpha,\beta)$ such that $\beta>\alpha^2-\alpha+1$. Then, for any edge $xy\in E$, we have $\kappa_{LLY}(x,y)=\frac{2+\alpha}{d}$, which is equivalent to say that there exists a perfect matching between $N_x$ and $N_y$.
\end{corollary}

\begin{proof}
Since $\beta>\alpha^2-\alpha+1$, we have $\frac{\beta-\alpha}{\beta-1}\alpha>\alpha-1$. Hence, we derive from Theorem \ref{thm:beta_alpha_1} that $\kappa_{LLY}(x,y)\ge\frac{2+\alpha}{d}$. By Theorem \ref{thm:upperk}, we deduce that $\kappa_{LLY}(x,y)=\frac{2+\alpha}{d}$, or equivalently, there exists a perfect matching between $N_x$ and $N_y$.
\end{proof}

\begin{remark}
    Several sufficient conditions involving restrictions on all parameters $n,d,\alpha,\beta$ for the existence of a perfect matching between $N_x$ and $N_y$ have been found in \cite[Theorem 1.4]{CLZ24}. In particular, a conjecture due to Bonini et. al. \cite{Bonini} on the existence of such local perfect matchings in conference graphs is confirmed. An advantage of Corollary \ref{cor:perfect_matching}, comparing to \cite[Theorem 1.4]{CLZ24}, is that we only need a restriction on $\alpha$ and $\beta$.
\end{remark}

\section{Bakry--\'Emery curvature formula of amply regular graphs}\label{section:BE}

In this section, we calculate the Bakry--\'Emery curvature at a vertex of an amply regular graph associated with either a balanced or an anti-balanced signature. For that purpose, we first recall the general curvature matrix method developed in \cite{CLP,CKLN,HL22,S20,S21}.

\subsection{Curvature matrix}

In \cite{HL22}, the second and third named authors establish the method for computing the Bakry--\'Emery curvature on connection graphs, i.e., graphs associated with an orthogonal or unitary matrix valued signature, building upon previous works \cite{CLP,CKLN,S20,S21} on unsigned graphs. We recall the calculation procedure here. For our later purpose, we restrict ourselves to the current setting of a $d$-regular graph associated with a signature $\sigma:E\to \{+1,-1\}$. For any two vertices $x,y$, let $a_{xy}=1$ if $xy\in E$ and $0$ otherwise.

Let $x\in V$ be a given vertex. For any integer $i$, define $S_i(x)$ to be the set of vertices at distance $i$ to $x$. Let us denote by 
\[S_1(x):=\{y_1,y_2,\ldots,y_d\},\,\text{and}\,\,S_2(x):=\{z_1,z_2,\ldots, z_{\ell}\}.\] 
For any $z_k\in S_2(x)$, we define its in-degree to be $d_{z_k}^{-}:=\sum_{i=1}^da_{y_iz_k}$.

In the case of $N=\infty$, the curvature dimension inequality (\ref{curvatureineq}) at a vertex $x$ becomes
\begin{equation*}
    \Gamma^{\sigma}_{2}(f)(x)-K\Gamma^{\sigma}(f)(x)\geq0,
\end{equation*}
for any function $f:V\rightarrow\mathbb{R}$. Observe that the value $\Gamma^{\sigma}_{2}(f)(x)$ depends only on the values of $f$ on $B_2(x)=\{x\}\cup S_1(x)\cup S_2(x)$. Let   $\Gamma^{\sigma}_{2}(x)$ be the symmetric matrix of size $(1+d+|S_{2}(x)|)\times (1+d+|S_{2}(x)|)$ corresponding to the quadratic bilinear form $\Gamma^{\sigma}_{2}(f)(x)$. 


Let us write $\Gamma^{\sigma}_{2}(x)$ blockwisely as \[\Gamma^{\sigma}_{2}(x)=\begin{pmatrix}
    \Gamma^{\sigma}_{2}(x)_{B_{1},B_{1}} & \Gamma^{\sigma}_{2}(x)_{B_{1},S_{2}}\\
    \Gamma^{\sigma}_{2}(x)_{S_{2},B_{1}} & \Gamma^{\sigma}_{2}(x)_{S_{2},S_{2}} 
\end{pmatrix},\]
where we use the notations that $B_1=B_1(x):=\{x\}\cup S_1(x)$ and $S_2=S_2(x)$.
A key observation is that the lower right block 
\[\Gamma^{\sigma}_{2}(x)_{S_{2},S_{2}}=\frac{1}{4}\mathrm{diag}(d_{z_{1}}^{-}, \ldots, d_{z_{\ell}}^{-}),\]
 is strictly positive definite. 
 Let $Q(x)$ be the Schur complement of $\Gamma^{\sigma}_{2}(x)_{S_{2},S_{2}}$ in the matrix $\Gamma^{\sigma}_{2}(x)$, that is, 
 \[Q(x):= \Gamma^{\sigma}_{2}(x)_{B_{1},B_{1}}-\Gamma^{\sigma}_{2}(x)_{B_{1},S_{2}}\left(\Gamma^{\sigma}_{2}(x)_{S_{2},S_{2}}\right)^{-1} \Gamma^{\sigma}_{2}(x)_{S_{2},B_1}. \]
Below we recall the explicit formula of the $(d+1)\times (d+1)$ matrix $Q(x)$ from \cite[(A.10)-(A.13)]{HL22}: The rows and columns of $Q(x)$ are indexed by $x$ and $y_{1},\dots, y_{d}$. For convenience, we multiply $Q(x)$ by a factor $4$.
\begin{align}    
    &4Q(x)_{x,x}=3d+d^{2}
    -\sum_{k=1}^\ell\frac{1}{d_{z_k}^-}\left(\sum_{i=1}^d a_{y_{i}z_{k}}\sigma_{xy_{i}}\sigma_{y_{i}z_{k}}\right)^2,\label{eq:2.1}\\
    &4Q(x)_{x,y_i}=\sum_{j=1,j\neq i}^d a_{y_jy_i}\sigma_{xy_j}\sigma_{y_jy_i}-2\left(d+1\right)\sigma_{xy_i}\notag\\
    &\hspace{5.5cm}+2\sum_{k=1}^\ell\frac{1}{d_{z_{k}}^{-}}\left(\sum_{j=1}^d a_{y_{j}z_{k}}\sigma_{xy_{j}}\sigma_{y_{j}z_{k}}\right)a_{y_{i}z_{k}}\sigma_{y_{i}z_{k}},\,\forall\,i,\label{eq:2.2}\\
    &4Q(x)_{y_{i},y_{i}}=\sum_{j=1,j\neq i}^d a_{y_jy_{i}}+2\left(d+1\right)-4\sum_{k=1}^\ell\frac{a_{y_{i}z_{k}}}{d_{z_{k}}^{-}},\,\forall \,i,\label{eq:2.3}\\
    &4Q(x)_{y_{i},y_j}=-4a_{y_iy_{j}}\sigma_{y_{i}y_{j}}+2\sigma_{xy_{i}}\sigma_{xy_{j}}-4\sum_{k=1}^\ell\frac{1}{d_{z_{k}}^{-}}a_{y_{i}z_{k}}a_{y_jz_{k}}\sigma_{y_{i}z_{k}}\sigma_{y_jz_{k}},\,\forall\,i,\,\forall\,j\neq i.\label{eq:2.4}
\end{align}

We define another symmetric matrix $B$ of size $(d+1)\times (d+1)$ as below (see \cite[(3.33)]{HL22})
\begin{align}\label{eq:B0}
B:=
\begin{pmatrix}
   1  & 1 & \cdots & 1 \\
      & 1 &        &   \\
      &   & \ddots &   \\
      &   &        & 1
\end{pmatrix}
\end{align}
We then calculate the product $2BQ(x)B^\top$  and write it blockwisely as 
\[2BQ(x)B^\top=\begin{pmatrix}
    a & \omega^\top \\
    \omega & (2BQ(x)B^\top)_{\hat{1}}
\end{pmatrix},\]
where $(2BQ(x)B^\top)_{\hat{1}}$ stands for the $d\times d$ matrix obtained from $2BQ(x)B^\top$ by removing the first row and the first column.
By \cite[Proposition A.1]{HL22}, it always holds that $a\geq 0$. Taking the Schur complement of $a$ in the matrix $2BQ(x)B^\top$, we arrive at the so-called curvature matrix \cite[Definition 3.1]{HL22}.
\begin{definition}[Curvature matrix]\label{def:curvature_matrix}
    Let $(G,\sigma)$ be a signed graph with $G=(V,E)$ being $d$-regular. The curvature matrix at a vertex $x\in V$ is defined as 
    \[A_{\infty}(G,\sigma,x):=(2BQ(x)B^\top)_{\hat{1}}-\omega a^{\dagger}\omega^{\top},\]
    where $a^{\dagger}=\frac{1}{a}$ if $a>0$ and $0$ if $a=0$.
\end{definition}
The Bakry--\'Emery curvature at a vertex of a signed graph is actually equal to the smallest eigenvalue of the curvature matrix.
\begin{theorem}[{\cite[Theorem 3.2]{HL22}}]\label{thm:curvature_matrix}
   Let $(G,\sigma)$ be a signed graph with $G=(V,E)$ being $d$-regular. Then we have
    \[K_{BE}(G,\sigma, x,\infty)=\lambda_{\min}(A_{\infty}(G,\sigma,x)),\]
    where $\lambda_{\min}(\cdot)$ stands for the smallest eigenvalue of the matrix.
\end{theorem}

\subsection{Bakry--\'Emery curvature formula with balanced signatures}

In this subsection, we derive the following formula for the Bakry--\'Emery curvature of an amply regular graph $G=(V,E)$. Note that $G$ can be considered as a signed graph $(G,\sigma)$ with $\sigma$ being the all-$+1$ signature.

\begin{theorem}\label{thm:ARGCurvature}
    Let $G=(V,E)$ be an amply regular graph with parameters $(n,d,\alpha,\beta)$. Let $x\in V$ be a vertex, and $A_{S_1(x)}$ be the adjacency matrix of the local graph induced by $S_1(x)$. Then we have 
    \begin{equation}\label{eq:curvatureformula}
    K_{BE}(+,x)=2+\frac{\alpha}{2}+\left(\frac{2d(\beta-2)-\alpha^2}{2\beta}+\frac{2}{\beta}\min_{\lambda\in \mathrm{sp} \left(A_{S_1(x)}\right)}\left(\lambda-\frac{\alpha}{2}\right)^2\right)_-,
    \end{equation}
    where the minimum is taken over all eigenvalues of $A_{S_1(x)}$,
    and $a_-:=\min\{0,a\}$, for any $a\in \mathbb{R}$.
    \end{theorem}


\begin{proof}
For $x\in V$, recall our notations $S_1(x)=\{y_1,\ldots, y_d\}$ and $S_2(x)=\{z_1,\ldots,z_\ell\}$.
Taking the signature $\sigma$ to be the all-$+1$ signature, we derive the matrix $4Q(x)$  from \eqref{eq:2.1}-\eqref{eq:2.4} as follows.
\begin{align*}
4Q(x)=
\begin{pmatrix}
    (4+\alpha) d & -4-\alpha & \cdots & -4-\alpha \\
    -4-\alpha &    &   &  \\
    \vdots   &    & P(x)  &   \\
    -4-\alpha &   &   &  
\end{pmatrix},
\end{align*}
where the $d\times d$ submatrix $P(x)$ is given by
\begin{align}
    (P(x))_{ii}=&2d+\alpha+2-\frac{4}{\beta}(d-\alpha-1), \,\,\,i=1,2,\ldots,d,\label{eq:Pii}\\ 
    (P(x))_{ij}=&2-4a_{y_iy_j}-\frac{4}{\beta}\sum_{k=1}^\ell a_{y_{i}z_k}a_{z_ky_{j}},\,\,\,i\neq j.\label{eq:Pij}
\end{align}
In the above, we have applied the fact that $d_{z_k}^{-}=\beta$ and \[\sum_{k=1}^\ell d_{z_k}^-=\sum_{i=1}^d(d-\alpha-1)=d(d-\alpha-1).\]
Recalling the definition of matrix $B$ in (\ref{eq:B0}), we calculate the product 
\begin{align}\notag
    4BQ(x)B^{\top}=
    \begin{pmatrix}
        0 & 0 & \cdots & 0 \\
        0 &   &        &   \\
        \vdots &   &  P(x) \\
        0 &   &        &
    \end{pmatrix}.
\end{align} 
By Definition \ref{def:curvature_matrix}, the matrix $\frac{1}{2}P(x)$ is the curvature matrix at $x$. Hence, due to Theorem \ref{thm:curvature_matrix}, we have 
\begin{equation}\label{eq:curvature_matrix}
    K_{BE}(+,x)=\frac{1}{2}\lambda_{\min}(P(x)).
\end{equation}
We observe that for $i\neq j$
\[\sum_{w\in V}a_{y_iw}a_{wy_j}=1+\left(A^2_{S_1(x)}\right)_{ij}+\sum_{k=1}^\ell a_{y_iz_k}a_{z_ky_j}=\left\{
  \begin{array}{ll}
    \alpha, & \hbox{if $a_{y_iy_j}=1$;} \\
    \beta, & \hbox{if $a_{y_iy_j}=0$.}
  \end{array}
\right.\]
That is, we have
\begin{equation}\label{eq:key_AS1}
    \sum_{k=1}^\ell a_{y_iz_k}a_{z_ky_j}=\alpha\left(A_{S_1(x)}\right)_{ij}+\beta\left(1-\left(A_{S_1(x)}\right)_{ij}\right)-1-\left(A^2_{S_1(x)}\right)_{ij}.
\end{equation}
Inserting \eqref{eq:key_AS1} into \eqref{eq:Pij} yields for $i\neq j$ that
\begin{equation}\label{eq:Pij_rearrangement}
    (P(x))_{ij}=-\frac{2(\beta-2)}{\beta}-\frac{4\alpha}{\beta}\left(A_{S_1(x)}\right)_{ij}+\frac{4}{\beta}\left(A^2_{S_1(x)}\right)_{ij}.
\end{equation}
Let us rearrange \eqref{eq:Pii} accordingly as 
\begin{equation}\label{eq:Pii_rearrangement}
      (P(x))_{ii}=\frac{2(\beta-2)(d-1)}{\beta}+\frac{4\alpha}{\beta}+(4+\alpha).
\end{equation}
Combining \eqref{eq:Pij_rearrangement} and \eqref{eq:Pii_rearrangement}, we obtain the curvature matrix
\begin{align*}
\frac{1}{2}P&=\left(\frac{\beta-2}{\beta}(d-1)+\frac{2\alpha}{\beta}+2+\frac{\alpha}{2}\right) \mathbf{I}_d+\frac{\beta-2}{\beta}(\mathbf{I}_d-\mathbf{J}_d)-\frac{2\alpha}{\beta}A_{S_1(x)}+\frac{2}{\beta}\left(A_{S_1(x)}^2-\alpha \mathbf{I}_d\right)\\
&=\left(\frac{d(\beta-2)}{\beta}+2+\frac{\alpha}{2}\right) \mathbf{I}_d-\frac{\beta-2}{\beta}\mathbf{J}_d-\frac{2\alpha}{\beta}A_{S_1(x)}+\frac{2}{\beta}A_{S_1(x)}^2,
\end{align*}
where $\mathbf{I}_d$ and $\mathbf{J}_d$ are the $d\times d$ identity matrix and  all-$1$ matrix, respectively, and we used the fact $\left(A_{S_1(x)}^2\right)_{ii}=\alpha$. 

Since the local graph induced by $S_{1}(x)$ is $\alpha$-regular, the constant vector $\mathbf{1}$ is an eigenvector of its adjacency matrix $A_{S_1(x)}$ with eigenvalue $\alpha$. It is then straightforward to check that 
\[\frac{1}{2}P(x)\mathbf{1}=\left(2+\frac{\alpha}{2}\right)\mathbf{1}.\]
The eigenvalues of $\frac{1}{2}P(x)$ corresponding eigenvectors vertical to the constant vector $\mathbf{1}$ are given by
\begin{align}\notag
2+\frac{\alpha}{2}+\frac{d(\beta-2)}{\beta}-\frac{2\alpha}{\beta}\lambda+\frac{2}{\beta}\lambda^2=2+\frac{\alpha}{2}+\frac{2d(\beta-2)-\alpha^2}{2\beta}+\frac{2}{\beta}\left(\lambda-\frac{\alpha}{2}\right)^2,
\end{align}
where $\lambda\in \mathrm{sp}(A_{S_1(x)}\vert_{\mathbf{1}^\perp})$ is any eigenvalue of $A_{S_1(x)}$ corresponding to an eigenvector vertical to $\mathbf{1}$.
By \eqref{eq:curvature_matrix}, we arrive at \begin{equation*}
    K_{BE}(+,x)=2+\frac{\alpha}{2}+\left(\frac{2d(\beta-2)-\alpha^2}{2\beta}+\frac{2}{\beta}\min_{\lambda\in \mathrm{sp}\left({A_{S_1(x)}}\vert_{\mathbf{1}^\perp}\right)}\left(\lambda-\frac{\alpha}{2}\right)^2\right)_-.
    \end{equation*} 
Let $\lambda_1=\alpha \geq \lambda_2 \geq \cdots \geq \lambda_{d}$ be the eigenvalues of $A_{S_1(x)}$. Then we have \[\mathrm{sp}\left({A_{S_1(x)}}\vert_{\mathbf{1}^\perp}\right)=\{\lambda_2,\ldots, \lambda_d\}.\] 
Recall that the second largest eigenvalue is non-negative for any non-complete connected graph, see, e.g., \cite[Lemma 1.7 (iii)]{Chung}. Therefore, whenever the local graph induced by $S_1(x)$ has a non-complete connected component, we have $\lambda_2\ge 0$, and hence
\begin{equation}\label{simple}
\min_{\lambda\in \mathrm{sp}\left({A_{S_1(x)}}\vert_{\mathbf{1}^\perp}\right)}\left(\lambda-\frac{\alpha}{2}\right)^2=\min_{\lambda\in \mathrm{sp}\left({A_{S_1(x)}}\right)}\left(\lambda-\frac{\alpha}{2}\right)^2.
\end{equation}
If the local graph induced by $S_1(x)$ has no non-complete connected components, it is a disjoint union of at least two complete graphs $K_{\alpha+1}$ of size $\alpha+1$, since $d\ge \alpha +2$ (see Definition \ref{ARG} and comments afterwards). Then, $\lambda_2= \alpha=\lambda_1$, and \eqref{simple} still holds.

    In conclusion, we simplify the formula for $K_{BE}(+,x)$ as follows:
\begin{equation*}
    K_{BE}(+,x)=2+\frac{\alpha}{2}+\left(\frac{2d(\beta-2)-\alpha^2}{2\beta}+\frac{2}{\beta}\min_{\lambda\in \mathrm{sp}\left({A_{S_1(x)}}\right)}\left(\lambda-\frac{\alpha}{2}\right)^2\right)_-.
    \end{equation*}
    This completes the proof.
\end{proof}


Consequently, we obtain the Bakry--\'Emery curvature of an amply regular graph with parameters $(n,d,\alpha,\beta)$, under various restrictions on the relation between $\alpha$ and $\beta$.

  
  


\begin{corollary}\label{cor:BEcurvature}
    Let $G=(V,E)$ be an amply regular graph with parameters $(n,d,\alpha,\beta)$. Let $x$ be a vertex.
    \begin{itemize}
        \item [(i)] If $\beta=1$, then $K_{BE}(+,x)=2+\frac{\alpha}{2}-d$.
        \item [(ii)] If  $1\neq \beta\geq \alpha$ and $(\alpha,\beta)\neq (2,2)$, then $K_{BE}(+,x)=2+\frac{\alpha}{2}$.
        \item [(iii)] If $\alpha=\beta=2$, then $2\leq K_{BE}(+,x)\leq 3$. Moreover, the upper bound is achieved if and only if the local graph induced by $S_1(x)$ is a disjoint union of at least two $3$-cycles, and the lower bound is achieved if and only if the local graph contains an $m$-cycle with $m$ being a multiple of $6$.
    \end{itemize}
\end{corollary}
\begin{remark}
A bit surprisingly, the Bakry--\'Emery curvature is uniquely determined by its parameters when $\beta\geq \alpha$, except for only one case $\alpha=\beta=2$. In the case of $\alpha=\beta=2$, the bound 2 and 3 is sharp for the Shrikhande graph and the $4\times 4$ Rooks' graph, respectively.
\end{remark}
\begin{proof}
Suppose $\beta=1$. Then the local graph induced by $S_1(x)$ is a disjoint union of complete graphs $K_{\alpha+1}$ of size $\alpha+1$ with at least two connected components. Therefore, the set of adjacency eigenvalues $\mathrm{sp}(A_{S_1(x)})$ consists of $\alpha$ and $-1$. Hence, we have 
\[\min_{\lambda\in \mathrm{sp}\left({A_{S_1(x)}}\right)}\left(\lambda-\frac{\alpha}{2}\right)^2=\frac{\alpha^2}{4}.\]
Therefore, we derive from \eqref{eq:curvatureformula} that $K_{BE}(+,x)=2+\frac{\alpha}{2}-d$ if $\beta=1$. This proves $(i)$.

Next, we show $(ii)$.
Notice that $d\geq \alpha+1$. Indeed, if $d=\alpha+1$, the graph has to be complete. Therefore, we have $d\geq \alpha+2$. If $\beta\geq \alpha\geq 3$, then we estimate
\[2d(\beta-2)-\alpha^2\geq 2(\alpha+2)(\alpha-2)-\alpha^2=\alpha^2-8\geq 0. \]
Therefore, we derive from the formula (\ref{eq:curvatureformula}) that
$K_{BE}(+,x)=2+\frac{\alpha}{2}$ if $\beta\geq \alpha\geq 3$.

If $\alpha=0$ and $\beta\geq 2$, we have $2d(\beta-2)-\alpha^2=2d(\beta-2)\geq 0$ and hence $K_{BE}(+,x)=2$. 

If $\alpha=1$ and $\beta\geq 2$, then the local graph induced by $S_1(x)$ is a $1$-regular graph, and hence must be a disjoint union of complete graphs $K_2$ of size $2$. If the local graph has only one connected component, then the graph $G$ is the complete graph $K_3$. Therefore, the local graph has at least two connected components. This leads to 
\[\min_{\lambda\in \mathrm{sp}\left({A_{S_1(x)}}\right)}\left(\lambda-\frac{\alpha}{2}\right)^2=\frac{1}{4}.\]
By \eqref{eq:curvatureformula}, we obtain $K_{BE}(+,x)=2+\frac{\alpha}{2}$ if $\alpha=1$ and $\beta\geq 2$.

If $\alpha=2$ and $\beta\geq 3$, we have $2d(\beta-2)-\alpha^2\geq 2d-\alpha^2>0$, and hence $K_{BE}(+,x)=2+\frac{\alpha}{2}$.
This completes the proof of $(ii)$.

Finally, let us consider the case $\alpha=\beta=2$. Then the formula \eqref{eq:curvatureformula} tells that 
\[K_{BE}(+,x)=3+\left(-1+\min_{\lambda\in \mathrm{sp} \left({A_{S_1(x)}}\right)}\left(\lambda-1\right)^{2}\right)_{-}.\] 
It follows that $2\leq K_{BE}(+,x)\leq 3$, where $K_{BE}(+,x)=3$ holds if and only if
\begin{equation}\label{eq:curvature_upper_sharp}
\min_{\lambda\in \mathrm{sp} \left({A_{S_1(x)}}\right)}\left(\lambda-1\right)^{2}\geq 1,
\end{equation}
and $K_{BE}(+,x)=2$ holds if and only if 
\begin{equation}\label{eq:curvature_lower_sharp}
\min_{\lambda\in \mathrm{sp} \left({A_{S_1(x)}}\right)}\left(\lambda-1\right)^{2}=0.
\end{equation}
Notice that the local graph induced by $S_1(x)$ is a $2$-regular graph, and hence a disjoint union of cycles. We first observe that the local graph can not contain a $4$-cycle, since otherwise there exist two vertices at distance $2$ from this $4$-cycle having at least $3$ common neighbors, contradicting the assumption that $\beta=2$.  

Recall that the eigenvalues of a cycle $C_n$ of $n$ vertices are $2\cos\frac{2j\pi}{n}, j=0,1,\ldots,\left\lfloor\frac{n}{2}\right\rfloor$. We have $K_{BE}(+,x)=3$, or equivalently \eqref{eq:curvature_upper_sharp} holds, if and only if the local graph is a disjoint union of at least two $3$-cycles.

We have  $K_{BE}(+,x)=2$, or equivalently \eqref{eq:curvature_lower_sharp} holds, if and only if the local graph has an eigenvalue $1$, which can only happen when the local graph contains a $m$-cycle such that $m$ is a positive multiple of $6$.
\end{proof}



\begin{corollary}\label{cor:BE_d}
Let $G=(V,E)$ be an amply regular graph with parameters $(n,d,\alpha,\beta)$. Then, the diameter of $G$ satisfies
\begin{itemize}
    \item [(i)]If $2d(\beta-2)\geq \alpha^2$, then \begin{align}\label{eq:diameterestimate}
    \mathrm{diam}(G)\leq \left\lfloor\frac{2d}{2+\frac{\alpha}{2}}\right\rfloor.
    \end{align}
  \item [(ii)]If $2d(\beta-2)\geq \alpha^2-\alpha\beta$, then  \begin{align}\label{eq:diameterestimate2}\mathrm{diam}(G)\leq d.\end{align}
\end{itemize}
\end{corollary}
\begin{proof}It's a direct consequence of 
 Theorem \ref{thm:ARGCurvature} and Theorem \ref{BEBM}.
\end{proof}
Now we are prepared to show Theorem \ref{thm:main1}, Corollary \ref{cor:QPK_weak} and Theorem \ref{thm:finite_ARG} from the Introduction.
\begin{proof}[Proof of Theorem \ref{thm:main1}]
By Corollary \ref{cor:BEcurvature} and Theorem \ref{BEBM}, we have $$\mathrm{diam}(G)\leq \left\lfloor\frac{2d}{2+\frac{\alpha}{2}}\right\rfloor.$$
Combining the above estimate with Theorem \ref{LLYd}, we obtain the desired conclusion.
\end{proof}

\begin{proof}[Proof of Corollary \ref{cor:QPK_weak}]
This is a direct consequence of Theorem \ref{thm:main1}. We mention that in the case $2=\beta=\alpha>\varepsilon d\geq 1$, we have by Corollary \ref{cor:BE_d} that $\mathrm{diam}(G)\leq d<2/\varepsilon<4/\varepsilon$.
\end{proof}

\begin{proof}[Proof of Theorem \ref{thm:finite_ARG}]
There are finitely many amply regular graphs with diameter $2$ and fixed degree. We suppose that $G$ has diameter at least three. If $d<3(\alpha-\beta+2)$, it follows by Theorem \ref{thm:1.2.3} that $G$ is the icosahedron or the line graph of a regular graph of girth at least $5$. If $G$ is the line graph of a regular graph of girth at least $5$, then $\beta =1$, which is contradictory to the assumption that $\alpha\le 6\beta-9$. 

Now, we suppose that $d\ge3(\alpha-\beta+2)$. If $\alpha\ge\beta$, since $6\beta-9-\alpha\ge0$, we have\begin{align}\notag
    &\frac{2d(\beta-2)-\alpha^2}{2\beta}\geq \frac{6(\alpha-\beta+2)(\beta-2)-\alpha^2}{2\beta}\\ \notag =&\frac{(6\beta-9-\alpha)(\alpha-\beta+3)+3-(\alpha+3)\beta}{2\beta}>-\frac{\alpha+3}{2}.
\end{align}
Then, Theorem \ref{thm:ARGCurvature} shows that $K_{BE}(+,x)>\frac{1}{2}$. Combining with Corollary \ref{cor:BEcurvature}, we have $K_{BE}(+,x)>\frac{1}{2}$ whenever $a\le 6\beta-9$. The positive curvature lower bound implies, by Theorem \ref{BEBM}, that $G$ is finite and has diameter at most $4d$. This completes the proof.
\end{proof}

\subsection{Bakry--\'Emery curvature formula with anti-balanced signatures}
In this subsection, we study the Bakry--\'Emery curvature of an amply regular graph associated with the all-$-1$ signature. By Proposition \ref{prop:switching}, this is equivalent to consider any anti-balanced signature on an amply regular graph.

\begin{theorem}\label{-1curvature}
    Let $(G,-)$ be a signed graph, where $G=(V,E)$ is an amply regular graph with parameters $(n,d,\alpha,\beta)$, and $-$ stands for the all-$-1$ signature. Let $x\in V$ be a vertex, and $A_{S_1(x)}$ be the adjacency matrix of the local graph induced by $S_1(x)$. Then we have
    \begin{equation}\label{-1inftycurvature}
        K_{BE}(-,x)=2+\left(\frac{2d(\beta-2)-\alpha^{2}}{2\beta}+\frac{5}{2}\alpha-2\beta+\frac{2}{\beta}\min_{\lambda\in\mathrm{sp}\left(A_{S_{1}(x)}\right)}\left(\lambda+\beta-\frac{\alpha}{2}\right)^2\right)_{-},
    \end{equation}
 where the minimum is taken over all eigenvalues of $A_{S_1(x)}$, and $a_-:=\min\{0,a\}$, for any $a\in \mathbb{R}$.
\end{theorem}
\begin{proof}
We first switch the all-$-1$ signature to be an anti-balanced signature $\sigma$ such that
\begin{align}
\sigma_{xy_{i}}&=+1,\,\,\text{for any $y_{i}\in S_{1}(x)$},\notag\\  
\sigma_{y_{i}y_{j}}&=-1,\,\,\text{for any edge $y_{i}y_{j}$, where $y_{i},y_{j}\in S_{1}(x)$},\notag\\
\sigma_{y_{i}z_{k}}&=+1,\,\,\text{for any edge $y_{i}z_{k}$, where $y_{i}\in S_{1}(x)$ and $z_{k}\in S_{2}(x)$}.\notag 
\end{align}
This can be achieved by choosing a switching function $\tau: V\to \{+1,-1\}$ such that $\tau(x)=+1, \tau\vert_{S_1(x)}\equiv -1$ and $\tau\vert_{S_2(x)}\equiv +1$. By Proposition \ref{prop:switching}, we have 
\begin{equation}\label{eq:switching_all_minus}
    K_{BE}(-,x)=K_{BE}(G,\sigma,x,\infty).
\end{equation}

Next we apply Theorem \ref{thm:curvature_matrix} to calculate $K_{BE}(G,\sigma,x,\infty)$. First, we derive from \eqref{eq:2.1}-\eqref{eq:2.4} the $(d+1)\times (d+1)$ matrix $4Q(x)$ as follows:
\begin{align*}
4Q(x)=
\begin{pmatrix}
(4+\alpha) d & -4-3\alpha & \cdots & -4-3\alpha \\
-4-3\alpha &   &   &   \\
\vdots &  &   R(x) &  \\
-4-3\alpha &   &   &   
\end{pmatrix},
\end{align*}
where $R(x)$ is a $d\times d$ submatrix given by
\begin{align*}
    (R(x))_{ii}&=2d+\alpha+2-\frac{4}{\beta}(d-\alpha-1),\,\,\,i=1,\dots,d.\\
    (R(x))_{ij}&=2+4a_{y_iy_j}-\frac{4}{\beta}\sum_{k=1}^\ell a_{y_{i}z_k}a_{z_ky_{j}},\,\,\,i\neq j.
\end{align*}
Recall the definition of matrix $B$, we calculate the product
\begin{align*}
    4BQ(x)B^{\top}=
    \begin{pmatrix}
        4\alpha d & 6\alpha & \cdots & 6\alpha \\
        6\alpha &  & &  & \\
        \vdots &  & R(x)   &   &  \\
        6\alpha &  &  &  &  \\
    \end{pmatrix}
\end{align*}
By Definition \ref{def:curvature_matrix}, we obtain the curvature matrix 
\begin{align}
2A_{\infty}(G,-,x)&=R(x)-
\begin{pmatrix}
    6\alpha \\
    6\alpha \\
    \vdots  \\
    6\alpha
\end{pmatrix}
\frac{1}{4\alpha d}
\begin{pmatrix}
    6\alpha & 6\alpha & \cdots & 6\alpha
\end{pmatrix}\notag\\
&=R(x)-\frac{9\alpha}{d}\mathbf{J}_d,\notag
\end{align}
where $\mathbf{J}_d$ is the $d\times d$ all-$1$ matrix. In the remaining part of the proof, we will write $A_\infty(x):=A_\infty(G,-,x)$ for short. 

Applying \eqref{eq:key_AS1}, we derive that for $i\neq j$
\begin{align*}
(R(x))_{ij}=-2+8\left(A_{S_1(x)}\right)_{ij}-\frac{4\alpha}{\beta}\left(A_{S_1(x)}\right)_{ij}+\frac{4}{\beta}+\frac{4}{\beta}\left(A_{S_1(x)}^{2}\right)_{ij}.
\end{align*}
Therefore, the curvature matrix is 
\begin{align*}
A_{\infty}(x)=\frac{2}{\beta}A_{S_1(x)}^{2}+\left(4-\frac{2\alpha}{\beta}\right)A_{S_1(x)}+\left(2+\frac{\alpha}{2}+d-\frac{2d}{\beta}\right)\mathbf{I}_{d}+\left(\frac{2}{\beta}-\frac{9\alpha}{2d}-1\right)\mathbf{J}_{d}.
\end{align*}

Then, we check for the constant vector $\mathbf{1}$ that
\begin{align*}
    A_{\infty}(x)\mathbf{1}=2\mathbf{1}.
\end{align*}
All the other eigenvalue of $A_{\infty}(x)$ corresponding to an eigenvector vertical to $\mathbf{1}$ is given by 
\begin{align*}
   &\frac{2}{\beta}\lambda^{2}+\left(4-\frac{2\alpha}{\beta}\right)\lambda +2+\frac{\alpha}{2}+d-\frac{2d}{\beta}\\
   =&2+\frac{2d(\beta-2)-\alpha^2}{2\beta}+\frac{5\alpha}{2}-2\beta+\frac{2}{\beta}\left(\lambda+\beta-\frac{\alpha}{2}\right)^2,
\end{align*}
where $\lambda\in \mathrm{sp}(A_{S_1(x)}\vert_{\mathbf{1}^\perp})$ is an eigenvalue of $A_{S_1(x)}$ corresponding to an eigenvector vertical to $\mathbf{1}$.
In conclusion, we have the Bakry--\'Emery curvature, by \eqref{eq:switching_all_minus} and Theorem \ref{thm:curvature_matrix},
\begin{align}\notag
      K_{BE}(-,x)=2+\left(\frac{2d(\beta-2)-\alpha^{2}}{2\beta}+\frac{5\alpha}{2}-2\beta+\frac{2}{\beta}\min_{\lambda\in\mathrm{sp}\left(A_{S_{1}(x)}|_{\mathbf{1}^{\perp}}\right)}\left(\lambda+\beta-\frac{\alpha}{2}\right)^2\right)_{-}.
\end{align}
By a similar argument in the proof of Theorem \ref{thm:ARGCurvature}, we arrive at
$$\min_{\lambda\in \mathrm{sp}\left({A_{S_1(x)}}\vert_{\mathbf{1}^\perp}\right)}\left(\lambda+\beta-\frac{\alpha}{2}\right)^2=\min_{\lambda\in \mathrm{sp}\left({A_{S_1(x)}}\right)}\left(\lambda+\beta-\frac{\alpha}{2}\right)^2,$$
and the desired result directly follows.
\end{proof}


\section{Applications in bounding eigenvalues, isoperimetric constants and volume growth of amply regular graphs}\label{section:furtherApp}
In this section, we discuss further applications of our discrete Ricci curvature estimates.


\subsection{Application in bounding the eigenvalues}\label{subsection:5.1}
In this subsection, we consider finite amply regular graphs. In fact, the finiteness will be guaranteed by the parameter conditions we impose. 
\begin{theorem}\label{thm:eigenvalueestimate}
Let $G=(V,E)$ be an amply regular graph with parameters $(n,d,\alpha,\beta)$. Let \[-d\leq\theta_{1}\leq\cdots\leq\theta_{n-1}\leq\theta_{n}=d\] be the eigenvalues of its adjacency matrix. 
\begin{itemize}
  \item [(i)]If $2d(\beta-2)\geq \alpha^2$, then \[\theta_{n-1}\leq d-2-\frac{\alpha}{2}.\] 
  \item [(ii)]If $2d(\beta-2)\geq \alpha(\alpha-\beta)$, then  \[\theta_{n-1}\leq d-2.\]
  \item [(iii)] If  $2d(\beta-2)\geq (\alpha-\beta)(\alpha-4\beta)$, then
  \[\theta_{1}\geq-d+2.\]
\end{itemize}
\end{theorem}
The Bakry--\'Emery curvature of a signed graph $(G,\sigma)$ is closely related to eigenvalue estimates for the Laplacian $\Delta^\sigma$. We say $\lambda^\sigma\in \mathbb{R}$ an eigenvalue of $\Delta^\sigma$ if there exists a non-zero function $f: V\to \mathbb{R}$ such that $\Delta^\sigma f+\lambda^\sigma f=0$. As a preparation for proving Theorem \ref{thm:eigenvalueestimate}, we recall the following Lichnerowicz type eigenvalue estimate.
\begin{theorem}[{\cite[Theorem 1.9]{LMP}}]\label{BEEI}
 Assume that a finite signed graph $(G,\sigma)$ satisfies $CD^{\sigma}(K,\infty)$ for $K\in \mathbb{R}$. Then, for any non-zero eigenvalue $\lambda^{\sigma}$ of $\Delta^{\sigma}$, we have $\lambda^{\sigma}\geq K$.
\end{theorem}

Recalling $\Delta^\sigma=A^\sigma-D$, the eigenvalues of the matrix $-\Delta^\sigma$ can be easily translated into eigenvalues of the signed adjacency matrix. We mention that the smallest eigenvalue of $\Delta^\sigma$ is $0$ if and only if $(G,\sigma)$ is balanced. Therefore, we have the following observations.
\begin{lemma}\label{lemma:LapAdj}
Let $(G,\sigma)$ be a finite signed graph with $G$ being $d$-regular. Let $\lambda^\sigma$ be the smallest non-zero eigenvalue of $\Delta^\sigma$ and $\theta_1\leq \cdots\leq\theta_{n-1}\leq \theta_n=d$ be the eigenvalues of adjacency matrix of the underlying graph $G$. 
 \begin{itemize}
    \item [(i)] If $\sigma$ is the all-$+1$ signature, then $\lambda^\sigma=d-\theta_{n-1}$.
    \item [(ii)] If $\sigma$ is the all-$-1$ signature,
    then $\lambda^\sigma=d+\theta_{1}$.
\end{itemize}   
\end{lemma}
The proof of the above lemma is straightforward, which we omit. 

\begin{proof}[Proof of  Theorem \ref{thm:eigenvalueestimate}]
    (i) Recall Theorem \ref{thm:ARGCurvature}. If $2d(\beta-2)-\alpha^2\geq 0$, then we obtain $K_{BE}(+,x)=2+\frac{\alpha}{2}.$ 
    Therefore, we derive by applying Theorem \ref{BEEI} and Lemma \ref{lemma:LapAdj} in the case of $\sigma\equiv+1$ that
    \[\theta_{n-1}=d-\lambda^\sigma\leq d-2-\frac{\alpha}{2}.\]
    (ii) If $2d(\beta-2)\geq\alpha^2-\alpha\beta$, then we have $\frac{2d(\beta-2)}{2\beta}\geq-\frac{\alpha}{2}$ and hence $K_{BE}(+,x)\geq2$. 
    Therefore, we derive by Theorem \ref{BEEI} and Lemma \ref{lemma:LapAdj} that
    \[\theta_{n-1}=d-\lambda^\sigma\leq d-2.\]
    (iii) Recall Theorem \ref{-1curvature}. If $2d(\beta-2)\geq(\alpha-\beta)(\alpha-4\beta)$, then
    \begin{align*}
        \frac{2d(\beta-2)-\alpha^2}{2\beta}+\frac{5\alpha}{2}-2\beta\geq 0,
    \end{align*}
   and therefore, $K_{BE}(-,x)=2$. 
     Applying Theorem \ref{BEEI} and Lemma \ref{lemma:LapAdj} with $\sigma\equiv -1$ gives
    \[\theta_{1}=-d+\lambda^\sigma\geq-d+2.\]
We complete the proof.
\end{proof}
Next, we prove Theorem \ref{thm:dual} from the Introduction. 
\begin{proof}[Proof of Theorem \ref{thm:dual}]
By Theorem \ref{thm:eigenvalueestimate}(iii), we only need to prove that$$(\alpha-\beta)(\alpha-4\beta)\le 2d(\beta-2).$$
One one hand, suppose that $2\le\alpha\le\beta$. By Theorem \ref{thm:1.9.3}, we have $d\ge2\beta$. An elementary  computation shows that 
$$(\alpha-\beta)(\alpha-4\beta)\le4\beta(\beta-2)\le 2d(\beta-2).$$
On the other hand, suppose that $\beta<\alpha\le10\beta-12$. If $d< 3(\alpha-\beta+2)$, Theorem \ref{thm:1.2.3} implies that $G$ is the icosahedron or the line graph of a regular graph of girth at least $5$. Since the diameter of $G$ is at least $4$, $G$ is not the icosahedron. If $G$ is the line graph of a regular graph of girth at least $5$, then $\beta=1$, which is contradictory to the assumption that $\alpha\le 10\beta-12$. Thus, we have $d\ge 3(\alpha-\beta +2)$. Then
\begin{align}\notag
(\alpha-\beta)(\alpha-4\beta)-2d(\beta-2) &\le (\alpha-\beta)(\alpha-4\beta)-6(\alpha-\beta +2) (\beta-2)\\ \notag &=(\alpha-\beta)(\alpha-10\beta+12)-12\beta+24\le0.
\end{align}
We complete the proof.
\end{proof}

Besides, a positive Lin--Lu--Yau curvature lower bound also gives a Licnherowicz type eigenvalue estimate. We state below the result restricting to the case of a regular graph.
\begin{theorem}[{\cite[Theorem 4.2]{LLY11}}]\label{LLYEI}
Let $G=(V,E)$ be a $d$-regular finite graph and $\lambda$ be a nonzero eigenvalue of the Laplacian $\Delta$. If $\kappa_{LLY}(x,y)\geq k>0$ for any edge $xy\in E$, then
\[\lambda\geq dk.\]
\end{theorem}
Combining both estimates derived from the Bakry--\'Emery curvature and Lin--Lu--Yau curvature bounds on amply regular graphs, we obtain Theorem \ref{thm:upper_bound_for_adjacency_eigencalue}.

\begin{proof}[Proof of Theorem \ref{thm:upper_bound_for_adjacency_eigencalue}]
(i) It follows by Theorem \ref{thm:beta_alpha_1} and Theorem \ref{LLYEI} that $$\theta_{n-1}\leq d-2-\left\lceil\frac{(\beta-\alpha)\alpha}{\beta-1}\right\rceil.$$
Corollary \ref{cor:BEcurvature}, Theorem \ref{BEEI}, and Lemma \ref{lemma:LapAdj} shows that $$\theta_{n-1}\leq d-2-\frac{\alpha}{2}.$$

(ii) The result is directly obtained from Corollary \ref{cor:BEcurvature}, Theorem \ref{BEEI}, and Lemma \ref{lemma:LapAdj}.
\end{proof}


\subsection{Application in bounding the 
 isoperimetric constants}

We first give the definitions of two important isoperimetric constants on graphs. For any two vertex set $A$ and $B$ of a graph $G=(V,E)$, let us denote by $|\cdot|$ the cardinality of a set, and by $|E(A, B)|:=\sum_{u\in A}\sum_{v\in B}a_{uv}$, where $a_{uv}=1$ if $uv\in E$ and $0$ otherwise. Notice that $|E(A,A)|$ is twice the number of edges with both ends in $A$.

\begin{definition}[Edge Cheeger constant]
    Let $G=(V,E)$ be a $d$-regular connected graph, the edge Cheeger constant $h$ is defined as
    \[h:=\inf_{\Omega\subset V,|\Omega|\leq\frac{1}{2}|V|}\frac{|E(\Omega,\Omega^{c})|}{d|\Omega|}.\]
\end{definition}
This constant is also called \emph{edge expansion}, measuring how fast this graph expand.
\begin{definition}
    [\cite{BJ13,Trevisan}, Bipartiteness constant]
    Let $G=(V,E)$ be a $d$-regular connected graph, the bipartiteness constant $\beta_0$ is defined as
    \[\beta_0:=\inf_{L,R\subset V,L\cap R =\emptyset}\frac{|E(L,L)|+|E(R,R)|+|\partial(L\cup R)|}{d|L\cup R|}.\]
    Here the notation $\partial(\Omega):=E(\Omega,\Omega^c)$ is the \emph{edge boundary set} of a subset $\Omega$.
\end{definition}
Notice that $\beta_0=0$ if and only if the graph $G$ is bipartite. The constant $\beta_0$ measures how far a graph is from being bipartite. The bipartite constant is introduced independently by Bauer and Jost \cite{BJ13} and Trevisan \cite{Trevisan}. Bauer and Jost call the constant $1-\beta_0$ dual Cheeger constant.

Both $h$ and $\beta_0$ are closed related to eigenvalues of Laplacians via so-called Cheeger and dual Cheeger inequalities, for which we refer to \cite{BJ13} and references therein. Translating into adjacency eigenvalues, there holds
$dh^2/2\leq d-\theta_{n-1}\leq 2dh,\,\,\text{and}\,\,d\beta_0^2/2\leq d+\theta_{1}\leq 2d\beta_0$.

Recalling the eigenvalue estimates on amply regular graphs from above, we further derive estimates for edge Cheeger and bipartiteness constants.
\begin{corollary}\label{cor:CheegerDual}
Let $G=(V,E)$ be an amply regular graph with parameters $(n,d,\alpha,\beta)$. 
\begin{itemize}
    \item [(i)] Suppose $1\neq \beta\geq \alpha$. Then for any subset $\Omega\subseteq V$ with $|\Omega|\leq \frac{1}{2}|V|$, we have 
    \[|E(\Omega, \Omega^c)|\geq |\Omega|.\]
    If we further assume $(\alpha,\beta)\neq (2,2)$, then
        \[|E(\Omega, \Omega^c)|\geq \left(1+\max\left\{\frac{\alpha}{2},\left\lceil\frac{(\beta-\alpha)\alpha}{\beta-1}\right\rceil\right\}\right)|\Omega|.\]
    \item [(ii)] Suppose $2d(\beta-2)\geq(\alpha-\beta)(\alpha-4\beta)$. Then for any subset $\Omega\subseteq V$, we have
        \begin{align*}
        \min_{L\sqcup R=\Omega}\left(|E(L,L)|+|E(R,R)|\right)+|E(\Omega,\Omega^{c})|\geq |\Omega|.
    \end{align*}
    where the minimization is taking over all possible partitions $\{L,R\}$ of $\Omega$, where one of $L$ and $R$ can be empty.
\end{itemize}
\end{corollary}
\begin{proof}
This follows directly from the fact $d-\theta_{n-1}\leq 2dh$ and $d+\theta_1\leq 2d\beta_0$ and the eigenvalue estimates proved in Subsection \ref{subsection:5.1}. 
\end{proof}

Now, we recall the definition of \emph{expander}.

\begin{definition}[{\cite[p.151]{AS16}}]
A graph $G = (V, E)$ is called an ($n, d, c$)-expander if it has $n$ vertices, the maximum degree of a vertex is $d$, and, for every set of vertices $W \subset V $of cardinality $|W| \le n/2$, the inequality $|N(W)| \ge c|W|$ holds, where $N(W)$ denotes the set of all vertices in $V \backslash W$ adjacent to some vertex in $W$.
\end{definition}

The following theorem mentioned in \cite[p.153]{AS16} shows a correspondence between the second largest eigenvalue of adjacency matrix of a graph and its expansion properties. For a proof, we refer to \cite[Theorem 4.3]{AM85}.

\begin{theorem}[{\cite[p.153]{AS16}}]\label{expander}
A $d$-regular graph $G$ with $n$ vertices is an ($n$, $d$, $c$)-expander for $c = 2(d-\theta_{n-1})/(3d-2\theta_{n-1})$.
\end{theorem}
Combining with our eigenvalue bound, we obtain the following conclusion.
\begin{theorem}\label{thm:expander}
 For an amply regular graph $G$ with parameters $(n,d,\alpha,\beta)$, if $1\ne \beta > \alpha$, then $G$ is an ($n$, $d$, $c$)-expander with
 $$c=1-\frac{d}{d+4+\max\left\{\alpha,2\left\lceil\frac{(\beta-\alpha)\alpha}{\beta-1}\right\rceil\right\}}.$$
\end{theorem}
\begin{proof}
It is a direct consequence of Theorem \ref{thm:upper_bound_for_adjacency_eigencalue} and Theorem \ref{expander} .
\end{proof}

\subsection{Application in bounding the volume growth}

\begin{theorem}\label{thm:Volume Growth}
    Let $G=(V,E)$ be an amply regular graph with parameters $(n,d,\alpha,\beta)$ such that $\kappa_{LLY}(x,y)\geq k >0$ for any $xy\in E$. Then, we have for any $x \in V$ and $i\geq1$ that
    $$|S_{i+1}(x)|\leq \frac{d-\max\left\{\alpha+1, ikd/2\right\}}{\beta}|S_{i}(x)|.$$
\end{theorem}

We need the following Lemma shown by Lin, Lu and Yau \cite[Lemma 4.4]{LLY11}. Let $x$ be a given vertex. For any $y\in S_i(x)$, $i\geq 1$, we use the following notations:
\[d^{x,+}_y:=|\{z\in S_{i+1}(x)\vert \,yz\in E\}| \,\,\text{and}\,\,d^{x,-}_y:=|\{z\in S_{i-1}(x)\vert\,zy\in E\}|.\]
\begin{lemma}\label{LLY4.3}
    For any two distinct vertices $x$ and $y$, we have $$\kappa_{LLY}(x,y)\leq \frac{1+(d^{x,-}_y-d^{x,+}_y)/d_y}{d(x,y)}.$$
\end{lemma}

\begin{proof}[Proof of Theorem \ref{thm:Volume Growth}]
Let $x$ be a given vertex. For any $y \in S_{i}(x)$ with $i\geq 1$, there exists a vertex $u \in S_{i-1}(x)$ such that $uy\in E$, and hence $u$ and $y$ share $\alpha$ common neighbors. These common neighbors can not fall into $S_{i+1}(x)$. Hence, we deduce that 
\begin{equation}\label{eq:d_alpha_1}
    d^{x,+}_y \leq d-\alpha-1. 
\end{equation}

On the other hand, for any $w\in S_{i+1}(x)$ with $i\geq 1$, there exists a vertex $v \in S_{i-1}(x)$ such that $d(v,w)=2$.  The two vertices $v$ and $w$ share $\beta$ common neighbors. These common neighbors all lie in $S_{i}(x)$. Hence we have \begin{equation}\label{eq:indegree_beta}
    d^{x,-}_w \geq \beta.
\end{equation}

Therefore, we estimate from $\sum_{w\in S_{i+1}(x)}d_w^{x,-}=E(S_i(x),S_{i+1}(x))=\sum_{y\in S_i(x)}d_y^{x,+}$ that
\begin{align}\label{volume1}
    |S_{i+1}(x)|\leq \frac{d-\alpha-1}{\beta}|S_{i}(x)|.
\end{align}

 Due to our curvature condition, it follows by Lemma \ref{LLY4.3} that $idk \leq d+d^{x,-}_y-d^{x,+}_y$.
    Since $d^{x,-}_y \leq d-d^{x,+}_y$, we have 
    \begin{equation}\label{eq:outdegree_k}
        d^{x,+}_y \leq \left(1-\frac{ik}{2}\right)d.
    \end{equation}
    Hence, by a similar argument as in \eqref{volume1}, we derive
    \begin{align}\label{volume2}
|S_{i+1}(x)| \leq \frac{\left(1-\frac{ik}{2}\right)d}{\beta}|S_{i}(x)|.
    \end{align}
    Inequalities \eqref{volume1} and \eqref{volume2} together yield the desired result.
\end{proof}

\begin{remark}\label{rmk:Seidel}
From the proof, it is straightforward to show that  Theorem \ref{thm:Volume Growth} holds still for graphs with maximal vertex degree $d$ such that any two adjacent vertices have at least $\alpha$ common neighbors and any two vertices at distance $2$ have at least $\beta$ common neighbors. A volume estimate for this class of graphs with diameter $2$ has been obtained by Seidel  \cite[Page 158]{Seidel} (see also \cite[Proposition 1.4.1]{BCN89}). Our results can be considered as a generalization of Seidel's volume estimate.
\end{remark}

To conclude this subsection, we prove Theorem \ref{thm:volume} from the Introduction. We need the following key lemma due to Terwilliger \cite[Lemma 3.3]{Terwilliger83} about $(s,c,a,k)$-graphs. Below we state a particular case of \cite[Lemma 3.3]{Terwilliger83} that $s=2$, i.e., the case of amply regular graphs, see also \cite[Proposition 1.9.1]{BCN89}.
\begin{lemma}\label{lemma:Terwilliger}
Let $G=(V,E)$ be an amply regular graph with parameters $(n,d,\alpha,\beta)$ such that $1\neq \beta\geq \alpha$. For any $x\in V$ and $w\in S_{i+1}(x)$ with $i\geq 1$, we have
\[d_w^{x,-}\geq \beta+i-1.\]
\end{lemma}

\begin{proof}[Proof of Theorem \ref{thm:volume}]
The proof is done by applying the procedure in the proof for Theorem \ref{thm:Volume Growth} via replacing \eqref{eq:indegree_beta} by 
\begin{equation}\label{eq:replace1}
    d_w^{x,-}\geq \beta+i-1, \,\,\text{for}\,\,w\in S_{i+1}(x)
\end{equation}
and replacing \eqref{eq:outdegree_k} by
\begin{equation}\label{eq:replace2}d_y^{x,+}\leq d-i\left(1+\frac{1}{2}\left\lceil\frac{\alpha(\beta-\alpha)}{\beta-1}\right\rceil\right),\,\,\text{for}\,\,y\in S_i(x).
\end{equation}
Notice that in \eqref{eq:replace1} and \eqref{eq:replace2}, we have applied Lemma \ref{lemma:Terwilliger} and Theorem \ref{thm:beta_alpha_1}.
\end{proof}

In fact, Theorem \ref{thm:volume} provides sharp volume estimate. The equality can be achieved for hypercube graphs. 
\begin{example}\label{ex:hypercube}
    For any $d\geq 2$, the $d$-dimensional hypercube graph $Q^d$ is amply regular with parameters $(2^d,d,0,2)$. Applying Theorem \ref{thm:volume} to $Q^d$ yields the estimates 
    \begin{equation}\label{eq:volume_hypercube}
        |S_{i+1}(x)|\leq \frac{d-i}{i+1}|S_i(x)|,\,\,\text{for}\,\,i\geq 1.
    \end{equation}
     Indeed, we have $|S_i(x)|=\binom{d}{i}$ for $1\leq i\leq \mathrm{diam}(Q^d)=d$. Thus, all the inequalities in \eqref{eq:volume_hypercube} are equalities. Therefore, the estimate for $|G|=\sum_{i\geq 0}|S_i(x)|$ derived from Theorem \ref{thm:volume} is also sharp for $Q^d$.
\end{example}

\section*{Acknowledgement}
  This work is supported by the National Key R \& D Program of China 2020YFA0713100 and the National Natural Science Foundation of China No. 12031017 and No. 12431004. We are very grateful to Shuliang Bai, Jack H. Koolen, Qianqian Yang for discussions on explicit constructions of amply regular graphs, especially in the case $\beta=1$.








\begin{thebibliography}{99}
\bibitem{AM85} N. Alon and V. D. Milman, $\lambda_1$, isoperimetric inequalities for graphs, and superconcentrators, J. Comb. Theory Ser. B 38 (1985), 73-88.
\bibitem{AS16} N. Alon and J. H. Spencer, The probabilistic method, 4th ed., John Wiley \& Sons, 2016.
\bibitem{BE} D. Bakry and M. \'Emery, Hypercontractivit\'e de semigroupes de diffusion, C. R. Acad. Sci. Paris S\'er. I Math. 299 (1984), 775-778.
\bibitem{BDKM15} S. Bang, A. Dubickas, J. H. Koolen and V. Moulton, There are only finitely many distance-regular graphs of fixed valency greater than two, Adv. Math. 269 (2015), 1-55.
\bibitem{BI} E. Bannai and T. Ito, Algebraic combinatorics I: Association schemes, The Benjamin/Cummings Publishing Co., Inc., Menlo Park, CA, 1984. 
\bibitem{BJ13} F. Bauer and J. Jost, Bipartite and neighborhood graphs and the spectrum of the
normalized graph Laplace operator, Comm. Anal. Geom. 21 (2013), no. 4, 787-845.
\bibitem{BRT21} B. Benson, P. Ralli and P. Tetali, Volume growth, curvature, and Buser-type inequalities in graphs, Int. Math. Res. Not. IMRN (2021), no.22, 17091-17139.
\bibitem{Bonini} V. Bonini, C. Carroll, U. Dinh, S. Dye, J. Frederick, E. Pearse, Condensed Ricci curvature of complete and strongly regular graphs, Involve, 13 (2020), no. 4, 559-576.
\bibitem{BCMLP} D. Bourne, D. Cushing, F. Münch, S. Liu and N. Peyerimhoff, Ollivier-Ricci idleness functions of graphs, SIAM J. Discrete Math. 32 (2018), no. 2, 1408-1424.
\bibitem{BCN89} A. E. Brouwer, A. M. Cohen and A. Neumaier, Distance-regular graphs, Springer-Verlag, 1989.
\bibitem{CL24} W. Chen and S. Liu, Curvature, diameter and signs of graphs, J. Geom. Anal. 34 (2024), no. 11, Paper No. 326.
\bibitem{CLZ24} K. Chen, S. Liu and H. Zhang, Curvature and local matchings of conference graphs and extensions, arXiv:2409.06418, 2024.
\bibitem{Chung} F. R. K. Chung, Spectral graph theory,
CBMS Regional Conf. Ser. in Math., 92, Published for the Conference Board of the Mathematical Sciences, Washington, DC; by the American Mathematical Society, Providence, RI, 1997. 
\bibitem{CKKLMP20} D. Cushing, S. Kamtue, J. H. Koolen, S. Liu, F. M\"unch, N. Peyerimhoff, Rigidity of the Bonnet-Myers inequality for graphs with respect to Ollivier Ricci curvature, Adv. Math. 369 (2020), 107188.
\bibitem{CKKLP} D. Cushing, S. Kamtue, R. Kangaslampi, S. Liu and N. Peyerimhoff, Curvature, graph products and Ricci flatness, J. Graph Theory 96 (2021), no. 4, 522-553.
\bibitem{CKLN} D. Cushing, S. Kamtue, S. Liu and N. Peyerimhoff, Bakry--\'Emery curvature on graphs as an eigenvalue problem, Calc. Var. Partial Differ. Equ. 61 (2022), no. 2, Paper No. 62.
\bibitem{CLP} D. Cushing, S. Liu and N. Peyerimhoff, Bakry--\'Emery curvature functions of graphs, Canad. J. Math. 72 (2020), no. 1, 89-143; arXiv:1606.01496.
\bibitem{Harary} F. Harary, On the notion of balance of a signed graph, Mich. Math. J. 2 (1953/54), 143-146.
\bibitem{Hiraki07} A. Hiraki, A characterization of the odd graphs and the doubled odd graphs with a few of their intersection numbers, Eur. J. Comb. 28 (2007), no. 1, 246-257.
\bibitem{HPS24} P. Horn, A. Purcilly and A. Stevens, Graph curvature and local discrepancy, J. Graph Theory (2024), early view.
\bibitem{HL22} C. Hu and S. Liu, Discrete Bakry--\'Emery curvature tensors and matrices of connection graphs, arXiv:2209.10762, 2022.
\bibitem{HL17} B. Hua and Y. Lin, Stochastic completeness for graphs with curvature dimension conditions, Adv. Math. 306 (2017), 279-302.
\bibitem{HLX24} X. Huang, S. Liu and Q. Xia, Bounding the diameter and eigenvalues of amply regular graphs via Lin--Lu--Yau curvature, Combinatorica (2024), online first.
\bibitem{KKRT} B. Klartag, G. Kozma, P. Ralli and P. Tetali, Discrete curvature and abelian groups, Canad. J. Math. 68 (2016), no.3, 655-674.
\bibitem{LL21} X. Li and S. Liu, Lin--Lu--Yau curvature and diameter of amply regular graphs, 	J. Univ. Sci. Tech. China 51 (2021), no. 12, 889-893.
\bibitem{LLY11} Y. Lin, L. Lu and S.-T. Yau, Ricci curvature of graphs, Tohoku Math. J. 63 (2011), no. 4, 605-627.
\bibitem{LY10} Y. Lin and S.-T. Yau, Ricci curvature and eigenvalue estimate on locally finite graphs, Math. Res. Lett. 17 (2010), no. 2, 343-356.
\bibitem{LW20} L. Lu and Z. Wang, On the size of planar graphs with positive Lin--Lu--Yau Ricci curvature, arXiv:2010.03716.
\bibitem{LMP18} S. Liu, F. Münch and N. Peyerimhoff, Bakry--\'Emery curvature and diameter bounds on graphs, Calc. Var. Partial Differ. Equ. 57 (2018), no. 2, Paper No. 67, 9 pp.
\bibitem{LMP} S. Liu, F. Münch and N. Peyerimhoff, Curvature and higher order Buser inequalities for the graph connection Laplacian, SIAM J. Discrete Math. 33 (2019), no. 1, 257-305.
\bibitem{LMP24} S. Liu, F. Münch and N. Peyerimhoff, 
Rigidity properties of the hypercube via Bakry--\'Emery curvature, Math. Ann. 388 (2024), no. 2, 1225-1259. 
\bibitem{Mulder79} M. Mulder, $(0,\lambda)$-graphs and $n$-cubes, Discrete Math. 28 (1979), 179-188.
\bibitem{MWAdvMath} F. M\"unch and R. Wojciechowski, Ollivier Ricci curvature for general graph Laplacians: heat equation, Laplacian comparison, non-explosion and diameter bounds, Adv. Math. 356 (2019), 106759, 45 pp.
\bibitem{NR17} L. Najman and P. Romon (Eds), Modern approaches to discrete curvature, Lecture Notes in Math., 2184, Springer, Cham, 2017.
\bibitem{NS22JCTB} A. Neumaier and S. Penji\'c, A unified view of inequalities for distance-regular graphs, part I, J. Comb. Theory Ser. B 154 (2022), 392-439.
\bibitem{NS22} A. Neumaier and S. Penji\'c, On bounding the diameter of a distance-regular graph, Combinatorica 42 (2022), no. 2, 237-251.
\bibitem{O09} Y. Ollivier, Ricci curvature of Markov chains on metric spaces, J. Funct. Anal. 256 (2009), no. 3, 810-864.
\bibitem{QPK19} Z. Qiao, J. Park and J. H. Koolen, On $2$-walk-regular graphs with a large intersection number $c_2$, Eur. J. Comb. 80 (2019), 224-235.
\bibitem{SalezGAFA} J. Salez, Sparse expanders have negative curvature, Geom. Funct. Anal. 32 (2022), no. 6, 1486-1513.
\bibitem{SalezJEMS} J. Salez, Cutoff for non-negatively curved Markov chains, J. Eur. Math. Soc. (JEMS) 26 (2024), no.11, 4375-4392.
\bibitem{Seidel} J. J. Seidel, Strongly regular graphs, in Surveys in Combinatorics, London Math. Soc. Lecture Note Ser., vol. 38, Cambridge Univ. Press, 1979, pp. 157–180.
\bibitem{S20} V. Siconolfi, Coxeter groups, graphs and Ricci curvature, S\'{e}m. Lothar. Combin. 84B (2020), Art. 67, 12 pp.
\bibitem{S21} V. Siconolfi, Ricci curvature, graphs and eigenvalues, Linear Algebra Appl. 620 (2021), 242-267.
\bibitem{Smith74} D. H. Smith, Bounding the diameter of a distance-transitive graph, J. Comb. Theory Ser. B 16 (1974), 139-144.
\bibitem{Smith14} J. D. H. Smith, Ricci curvature, circulants, and a matching condition, Discrete Math. 329 (2014), 88-98.
\bibitem{Terwilliger82} P. Terwilliger, The diameter of bipartite distance-regular graphs, J. Comb. Theory Ser. B 32 (1982), 182-188.
\bibitem{Terwilliger83} P. Terwilliger, Distance-regular graphs and $(s,c,a,k)$-graphs, J. Comb. Theory Ser. B 34 (1983), 151-164.
\bibitem{Trevisan} L. Trevisan, Max cut and the smallest eigenvalue, SIAM J. Comput. 41 (2012), no. 6, 1769-1786.
\bibitem{Zaslavsky} T. Zaslavsky, Signed graphs, Discrete Appl. Math. 4 (1982), no. 1, 47-74.


\end{thebibliography}
\end{document}